\documentclass[12pt]{amsart}
\usepackage{latexsym,fancyhdr,amssymb,color,amsmath,amsthm,graphicx,listings,comment}
\usepackage[section]{placeins}
\newtheorem{thm}{Theorem} 
 
\newtheorem{propo}{Proposition} 
\newtheorem{lemma}{Lemma} 
\newtheorem{coro}{Corollary}
\let\paragraph\subsection

\title{Dehn Sommerville Manifolds}
\author{Oliver Knill}
\date{August 19, 2025}

\address{Department of Mathematics \\ Harvard University \\ Cambridge, MA, 02138 }
\makeindex


\begin{document}
\maketitle

\begin{abstract}
Dehn-Sommerville manifolds are a class of finite abstract simplicial complexes
that generalize discrete manifolds. Despite a simpler definition in comparison to manifolds, they
still share most properties of manifolds. They especially satisfy all Dehn-Sommerville 
symmetries telling that half of the $f$-vector entries are redundant. They also share other properties with
$q$-manifolds: for every Dehn-Sommerville $q$-manifold $G$ and any function 
$g: V(G) \to A_k=\{0, \dots, k\}$ with $k \geq 1$, the set of $x$ such that $A_k \subset g(x)$
is a Dehn-Sommerville $(q-k)$-manifold if not empty. We also see
that for Dehn-Sommerville $q$-manifolds, all higher characteristics $w_m(G)$ agree with Euler characteristic
$w_1(G)=\chi(G)$. We also see that the chromatic number of a Dehn-Sommerville q-manifold is bounded above by $2q+2$ and that 
odd-dimensional Dehn-Sommerville manifolds are flat
and form a monoid under the join operation. In general, Dehn-Sommerville manifolds are 
invariant under edge refinement, Barycentric refinement and Cartesian products.  
\end{abstract}

\section{Introduction}

\paragraph{}
The {\bf Dehn-Sommerville symmetry} \cite{Dehn1905,Sommerville1927,Klee1964} 
(see also \cite{gruenbaum,BergerLadder, BayerBillera,NovikSwartz,MuraiNovik,LuzonMoron, BrentiWelker,Hetyei,Klain2002,Ziegler})
can best be understood by introducing a class of finite abstract simplicial complexes which all share this symmetry. 
The symmetry states that half of the combinatorial data $f_k$ counting $k$-dimensional part of space are redundant.
This class, where this works includes manifolds but goes much beyond the usual manifolds. The class of simplicial complexes 
is here called the class of {\bf Dehn-Sommerville q-manifolds}. In order to appreciate this, one should note that most literature in the
context of Dehn-Sommerville is about convex polyhedra \cite{gruenbaum,Ziegler} (and so about discrete spheres).
The difficulties of defining polyhedra is a fascinating showcase in the general epistemology of science
\cite{Gruenbaum1970,lakatos,gruenbaum,Gruenbaum2003}. 
The fact that Dehn-Sommerville identities like $-7898 f_1 + 11847 f_2 -14360 f_3 + 16155 f_4, -17528 f_5 + 18627 f_6=0$
hold for all 6-manifolds (and similar explicit identities for arbitrary dimensions) 
seems not have been realized before  \cite{dehnsommervillegaussbonnet}. But there are
three such identities and they hold for all $6$-manifolds. They hold even for more general 6-dimensional objects like 
the suspension $G$ of any 5-manifold $M$, which is no more a manifold if $M$ is not a 5-sphere. The suspension of a 
non-sphere 5-manifold is an example of a Dehn-Sommerville 6-manifold that is not a 6-manifold. It still has the 
Dehn-Sommerville symmetry. 
%  G=JoinAddition[Whitney[CompleteGraph[{2,2,2}]],Whitney[CompleteGraph[{2,2,2,2}]]]; FVector[G].Y[2,6] 
\index{Polyhedron}
\index{Simplicial complex}
\index{Dehn-Sommerville symmetry}

\paragraph{}
The definition is as follows: a {\bf Dehn-Sommerville q-manifold} is a $q$-dimensional finite
abstract complex for which all unit spheres $S(x)= \delta U(x)$, (the topological boundaries of the smallest open 
set $U(x)=\{y \in G, x \subset y\}$ containing $x \in G$ in the complex equipped with the finite 
{\bf Alexandrov topology}
\cite{Alexandroff1937,KnillTopology2023}) are Dehn-Sommerville $(q-1)$-manifolds of Euler characteristic 
$1-(-1)^q$. By the {\bf Euler gem formula} for the Euler characteristic of traditional discrete spheres, every 
discrete q-manifold is a Dehn-Sommerville manifold. The sphere formula $\sum_{x \in G} \omega(x) \chi(S(x))$ 
\cite{Sphereformula} (holding for all simplicial complexes) shows that every 
odd-dimensional Dehn-Sommerville manifold is a Dehn-Sommerville sphere. For any Dehn-Sommerville 
$q$-manifold, there is a $[(q+1)/2]$ dimensional space of valuations that vanish. In other words,
half of the {\bf combinatorial f-data} of any manifold are redundant. 
And this is independent of whether 
we have {\bf  Poincar\'e duality} or not. Most Dehn-Sommerville $q$-manifolds do not 
satisfy the Poincar\'e duality - even if they are orientable. 
In the next few paragraphs, we sketch the various pictures. Some of these topics are 
reviewed later in more detail. As mentioned in the abstract, there are also new results.
\index{Dehn-Sommerville manifold}
\index{f-data}
\index{Alexandrov topology}
\index{finite topology}
\index{Euler gem formula}
\index{Poincar\'e duality}

\paragraph{}
{\bf Picture 1}: first of all, there are the historical
valuations $X_{k,q}$ on simplicial complexes. Every {\bf valuation} (a functional $X$ defined on the set of
subsets of $G$ satisfying $X(A \cup B) +X(A \cap B)=X(A)+X(B)$) has its own {\bf Gauss-Bonnet formula} 
$X(G) = \sum_{v \in V} K(v)$. It writes $X(G)$
as a sum of curvatures $K(v)$ on vertices $v \in V$, 
where $V$ are the set zero-dimensional elements in $G$. (The dimension of a set $x \in X$ is $|x|-1$
where $|x|$ is the cardinality.)
The Gauss-Bonnet curvatures are special because the curvature at a vertex $v$ agrees with
$X_{k-1,q-1}$, applied to the unit sphere $S(v)$: we have 
$X_{k,q}(G) = \sum_{v \in V} X_{k-1,q-1}(S(v)$.
The valuation $X_{-1,q}=\chi$ is the {\bf Euler characteristic}. It is zero in the
odd-dimensional case, implying $X_{0,q}$ is zero in even dimension $q$. 
So, all $X_{k,q}$ are zero for $k \geq 0$. One can then see that half of them are independent. 
The space of valuations on a simplicial complex of dimension $q$
has dimension $q+1$ by the {\bf discrete Hadwiger theorem} \cite{KlainRota}. The functionals 
$G \to f_k(G)$ form a basis for $k=0, \dots, q$ if the maximal dimension is $q$. 
\index{Hadwiger theorem}
\index{discrete Hadwiger theorem}
\index{dimension of a simplex}
\index{Gauss-Bonnet formula}
\index{valuation}

\paragraph{}
{\bf Picture 2}: the second picture explains how many of these invariants are redundant. 
The {\bf simplex generating function} $f_G(t)=\sum_{k=-1}^q f_k t^{k+1}$, 
(where the convention $f_{-1}=1$ is used for notational reason), encodes the combinatorial
{\bf $f$-vector} $(f_0,f_1,\dots, f_k)$ of $G$. 
The {\bf functional Gauss-Bonnet theorem} $f'_G=\sum_v f_{S(v)}$  \cite{dehnsommervillegaussbonnet}
which holds for every simplicial complex $G$, implies that Dehn-Sommerville q-manifolds have the 
property that $f_G$ is either even or odd with respect to the point $t_0=-1/2$, 
depending on whether the maximal dimension $q$ of $G$ is even or odd.
By looking at roots, this can be seen to be equivalent to the statement that the 
polynomial $h(t)=(t-1)^d f(1/(t-1))$ has {\bf palindromic coefficients}.
The later polynomial property in turn shows that $[(q+1)/2]$ valuations are redundant. 
This implies that at least half of the $q+1$ eigenvectors $Y_{k,q}$ of the transpose 
{\bf Barycentric refinement operator} $A_q$ must be zero. It will imply that 
at least half of all the classical valuations $X_{k,q}$ are linearly independent. 
\index{simplex generating function}
\index{f-vector}
\index{combinatorial f-data}
\index{functional Gauss-Bonnet}
\index{Barycentric refinement operator}
\index{Barycentric eigenvectors}

\paragraph{}
{\bf Picture 3}: the third picture uses {\bf Barycentric refinement}.
The Dehn-Sommerville symmetry is a {\bf duality}. It is a combinatorial analog of {\bf Poincar\'e duality}
but it is rather complementary: most Dehn-Sommerville manifolds do not satisfy Poincar\'e duality
for cohomology.  Let us look at an example which was noted in
\cite{valuation,dehnsommervillegaussbonnet}: the $f$-vector of 
{\bf every 4-manifold} satisfies $-22 f_1 + 33 f_2 -40 f_3 + 45 f_4=0$.
The classical Dehn-Sommerville story as developed by the pioneers only assumes 
simplicial complexes to be ``polyhedra", assuming in particular that we have a sphere
manifold structure. The class of Dehn-Sommerville manifolds is much larger: 
any suspension of an arbitrary 3-manifold for example is a Dehn-Sommerville 4-manifold. 
There is also no mystery about the valuations. The vector $Y_{2,4}=[0, -22, 33, -40, 45]$ is an eigenvector of 
the Barycentric refinement operator (a $5 \times 5$ matrix) in dimension $q=4$. 
The other linearly independent Dehn-Sommerville invariant in dimension $q=4$ is the trivial valuation
$[0, 0, 0, -2, 5]$ and eigenvector of $A^T$ which restates that 5 times the number of 4-simplices is twice the
number of 3-simplices. It is rephrasing that every 3-simplex in a 4-manifold 
is contained in exactly two 4-simplices. This manifold property was the only one that was necessary 
to define geodesic flow and sectional curvature. 
The work \cite{geodesics1,geodesics2} was the reason to revisit the Dehn-Sommerville story here. 
\index{geodesics}
\index{symmetries for manifolds}
\index{duality}
\index{Barycentric refinement}

\paragraph{}
{\bf Picture 4}: the relation with the discrete version of the Gauss-Bonnet-Chern curvature provides an other picture.
In \cite{cherngaussbonnet} already, we had simplified the curvature of a 5 manifold as 
$K(v) = -f_1(S(v))/6 + f_2(S(v))/4 - f_3(S(v))/6$. We had used there that the unit sphere $S(v)$
has the Euler characteristic $f_0-f_1+f_2-f_3=f_4=2$ and that 
$2 f_4 = 6 f_5$ in the unit sphere, but we could not yet verify then (back in 2010) that this curvature was
always zero for all 5-manifolds. Today we would see this curvature as a linear combination of classical 
Dehn-Sommerville invariants like $K(v) = X_{1,4}/12 -X_{3,4}/6$ and so
verify that it is zero. As pointed out in \cite{indexformula} based on \cite{poincarehopf,indexexpectation,
PoincareHopfVectorFields,MorePoincareHopf}, it
is much simpler however to see curvature $K(v)$ as the expectation of symmetrized {\bf Poincar\'e-Hopf indice }
$j_g(v)/2 = [i_g(v)+i_{-g}(v)]/2$ which for odd-dimensional Dehn-Sommerville $q$-manifolds $G$ 
is equal to $-\chi(M_g(v))/2$, where $M_g(v)$ is the level surface 
$\{ g=c \} = \{ w, g(w) = c=g(v)\}$ in the even dimensional manifold $S(v)$ which 
is the Euler characteristic of an odd-dimensional manifold and so zero (we discuss level surfaces
later on). It is here that our new result that level surfaces in Dehn-Sommerville q-manifolds are Dehn-Sommerville
$(q-1)$-manifolds comes in. The complete picture of valuations in Dehn-Sommerville manifolds illustrates
the Gauss-Bonnet story, that started the research in \cite{cherngaussbonnet}.
\index{Gauss-Bonnet-Chern curvature}
\index{index expectation}
\index{Poincar\'e-Hopf index}
\index{level surface}

\paragraph{}
{\bf Picture 5}: there is an other line of research which fits into the Dehn-Sommerville story. This is the 
{\bf connection calculus story} that has lead to {\bf higher characteristic invariants} $w_m(G)$ 
\cite{CharacteristicTopologicalInvariants}, where Euler characteristic $\chi(G)=w_1(G)$ is the 
first invariant. These invariants go back to \cite{Wu1953,Gruenbaum1970} and are defined for 
any subset $A \subset G$ and given by $w_m(A) = \sum_{x \in A^m, \bigcap_j x_j \in A} \prod_{j=1}^m \omega(x_j)$,
where $\omega(x_j)= (-1)^{{\rm dim}(x_j)}$ and $U(x)=\{ y, x \subset y\}$ is the star of $x$. These invariants
parallel the first characteristic = Euler characteristic, where simplicial cohomology satisfies the {\bf Euler-Poincar\'e}
and the sphere formula $\sum_{x \in G} \omega(x) \chi(S(x))=0$ hold. We have seen
that in general, the {\bf sphere formula} $\sum_{x \in G} \omega(x) w_m(S(x))=0$ holds 
\cite{Sphereformula}. This immediately implied that all odd-
dimensional Dehn-Sommerville manifolds $G$ have zero higher invariants $\omega_m(G)=0$. 
The connection calculus story is quite general \cite{Unimodularity,GreenFunctionsEnergized,EnergizedSimplicialComplexes,
EnergizedSimplicialComplexes2,EnergizedSimplicialComplexes3}.
\index{connection calculus}
\index{higher characteristics}
\index{Euler characteristic} 
\index{Euler-Poincar\'e}
\index{sphere formula}

\paragraph{}
As we will see here, more is true. This is something that goes beyond review and is new.
All higher characteristic $w_m(G)$ are the same
$w_m(G)=\chi(G)$ if $G$ is a Dehn-Sommerville q-manifold. We knew that before only for 
$q$-manifolds, special Dehn-Sommerville manifolds in which the unit spheres are classical spheres. The subject of
higher characteristic invariants is exciting. Each of these quantities can be seen as a total potential
theoretic energy. There are {\bf k-point Green function identities}: consider any $k$-tuple of points $X \in G^k$
which do not necessarily have to intersect. 
The {\bf m'th order potential energy} of the {\bf k-particle configuration} $X$ is defined as $V(X)=w(X) w_m(U(X)$,
where $w(X)=\prod_{j=1}^k \omega(x_j)$ and $U(X)=\bigcap_{j=1}^k U(x_j)$, and where $U(x_j)$ are the
{\bf stars}, the smallest open sets in the Alexandrov topology $\mathcal{O}$ of the complex. The topology 
$\mathcal{O}$ is a finite topology that is non-Hausdorff in positive dimensions, similarly to the Zariski topology 
in algebraic geometry. One of the theorems proven in \cite{CharacteristicTopologicalInvariants} is that the 
{\bf total $k$-point Green function energy} is equal to the total energy 
$\sum_{X \in G^k} w(X) w_m(U(X))=w_m(G)$. For $k=1$ already, the formula  $\sum_{x \in G} w(x) w_m(U(x))=w_m(G)$ is
very useful as it is a fast way to compute the invariant $w_m$ as a sum of curvatures = potential energies
of elements $x \in G$. For $k=2$, where the potential energy 
$g(x,y) = V(x,y)$ between two particles $x,y$ defines a $n \times n$ matrix $g$ (where $n$ is the number of elements in $G$),
$g$ is the inverse of the {\bf connection Laplacian matrix} $L$. The discovery of these identities for $k=1$ and $k=2$ 
is documented in in \cite{Unimodularity,GreenFunctionsEnergized,KnillEnergy2020}. It was generalized to larger $k$'s 
in \cite{CharacteristicTopologicalInvariants}. 
\index{Green functions}
\index{k-point Green function}
\index{k-particle configuration}
\index{Zariski topology}
\index{Alexandrov topology}
\index{potential energy}
\index{connection matrix}
\index{connection Laplacian}
\index{unimodularity}

\paragraph{}
{\bf Picture 6}: the last picture is a relation with the {\bf arithmetic of geometric objects}. Dual to the 
{\bf disjoint union} of geometries is the {\bf join} operation 
$A \oplus B = A \cup B \cup \{ a \cup b, a \in A, b \in B\}$
which generates from two complexes of dimension $a,b$ a new object of dimension $a+b+1$. The {\bf zero} 
element is the {\bf void} $0=\{\}$, the empty complex (which of course is again a simplicial complex as 
it is closed under the subset operation of its elements and because it does not contain the empty set).
A {\bf group completion} defines then the Abelian group of {\bf signed complexes}. 
The join operation is useful because joining two spheres produces a new sphere. One can so for example
upgrade a given $q$-sphere $G$ to a $(q+1)$-sphere $G \oplus S^0$ by joining the $0$-sphere 
$S^0=\{ \{1\},\{2\} \}$. This is called the {\bf suspension}. There is also the {\bf Euler gem monoid} 
consisting of simplicial complexes, satisfying the {\bf Euler Gem formula} $\chi(G)=1+(-1)^q$. 
Then there is the monoid of {\bf varieties} which are inductively
defined as complexes $G$ with the property that all $S(x)$ are varieties of dimension $1$ less.
Also the intersection of two monoids is always again a monoid. The {\bf Dehn-Sommerville sphere monoid}
is the intersection of {\bf variety monoid} with the {\bf Euler-gem monid}.
Since all odd-dimensional manifolds are Dehn-Sommerville spheres, the monoid of 
Dehn-Sommerville spheres contains also the submonoid of all odd-dimensional Dehn-Sommerville
manifolds. We therefore can ``calculate" with all odd-dimensional Dehn-Sommerville manifolds. What
is amazing is that Dehn-Sommerville manifolds share so many properties of actual manifolds, 
despite their simpler definition and despite that they are much more general. 
% Does the Sabidussy ring extend to simplicial complexes? 
\index{arithmetric}
\index{Euler-Gem monoid}
\index{arithmetic of graphs}
\index{void}
\index{suspension}
\index{join operation}
\index{monoid}

\paragraph{}
Of course, we should also wonder about the {\bf physical relevance} of Dehn-Sommerville manifolds. 
Every physical measurement that has been performed so far indicates that our physical space is a 
3-dimensional space. No single measurement has indicated that our physical space has spheres $S_r(x)$ that are not
topological $2$-spheres or that is a higher dimensional fiber bundle in which the fibers are too small
to be detected. Attempts to merge quantum mechanics and general relativity suggest that non-trivial topologies
could occur at the {\bf Planck scale} $10^{35} {\rm meters}$ but such scales are many order 
of magnitudes off from experimental reach. There have been {\bf speculations} 
(comparable to Demokritus' guesses) that matter is quantized. The evidence of ancient philosophers,  
was to look at the smallest possible particles like gypsum and speculate from limitations that it is 
too small to be further divided. It turned out to be false; only from the 19th century on (Brownian motion) 
revealed that atoms have a scale of $10^{-10}$ meters. They themselves turned out have structure, 
to be divisible. Whether the smallest known quark constituents of matter at a scale of $10^{-22}$ are 
divisible (like having a preon structures) is again only a playground for speculations. 
The question whether physical space can be of Dehn-Sommerville nature 
is interesting. A simple starting point is the question
to classify Dehn-Sommerville $q$-spheres. Every connected Dehn-Sommerville 2-sphere is a 
2-sphere. There are plenty of disconnected Dehn-Sommerville 3-spheres, like a suspension of a disjoint union 
$M$ of 2-manifolds whose sum of Euler characteristics $\chi(M)$ is $2$ so that $M$ is a Dehn-Sommerville 
$2$-sphere.
\index{Demokritus's guess}
\index{Planck scale}
\index{preon structure}

\paragraph{}
The {\bf Vietoris-Ribbs picture} is to take an continuum compact Riemannian $q$-manifold $M$ and
pick $n$ random points on $M$. Given $\epsilon>0$ and a point cloud in $M$
that is $\sqrt{\epsilon}$ dense. Define the graph in which the $n$ points are the vertices and 
where two points $v,w$ are connected if $d(v,w)<\epsilon$. The Whitney complex of this graph is an example of
a {\bf Vietoris-Ribbs complex}. Its dimension goes to infinity if $\epsilon \to 0$ and the number of points is
adapted accordingly. The $k$-simplices $x$ of this complex consist of $k+1$ points which all have distance less
than $\epsilon$ from each other. Now start {\bf melting points away} by successively removing 
vertices from the graph for which $S(v)$ is contractible. What we tried unsuccessfully to prove 
\footnote{See our talk "Homotopy manifolds from   May 3, 2021}, we end up with a 
discrete q-manifold. We could prove this only in dimension $1$, where it is related to the 
{\bf "machine graph"},\footnote{See our talk "The machine graph" from Oct 2, 2022}, 
where $V$ is the finite set of the real numbers represented by a computer and $\epsilon$ is 
determined by machine precision, same tolerance and equal tolerance of the implemented arithmetic
(for simplicity, we can identify the largest and smallest possible machine number to be topologically on a circle
to have no boundary). But dimension $q=1$ is special in that classically it is the only case, where unit spheres
are disconnected. Because every connected Dehn-Sommerville 1-manifold is a circle, the failure to extend the
Vietoris-Rips conjecture to higher dimensions $q>1$ could also be founded on the fact that the
reduction process could end up not in a q-manifold but Dehn-Sommerville q-manifold. It could
be the case that in higher dimensions, the melting process of removing points with contractible 
unit spheres ends up with a Dehn-Sommerville (q-1) spheres and not with (q-1) sphere. We still think that the 
original conjecture has a chance to be true, especially in small dimensions like $q=2$, where the Jordan curve theorem kicks in
preventing unit spheres to contain a disconnected union of circles. In any case, an easier task is
to ask whether that every Vietoris-Ribbs graph of a compact Riemannian q-manifold $M$ defines for large enough 
$n$ a finite Dehn-Sommerville q-manifold obtained, after successfully melting away points $v$ which have
contractible unit spheres $S(v)$. 
\index{Vietoris-Ribbs complex}
\index{Machine graph}
\index{machine precision}

\section{Dehn-Sommerville manifolds}

\paragraph{}
A {\bf finite abstract simplicial complex} is a finite set of non-empty 
sets closed under the operation of taking non-empty subsets \cite{DehnHeegaard}. Every
such geometry $G$ carries a {\bf finite topology} $\mathcal{O}$ \cite{Alexandroff1937,May2008,KnillTopology} 
in which the {\bf open sets} are unions of stars $U(x) = \{ y \in G, x \subset y\}$ and where 
the {\bf closed sets} agree with the sub-simplicial
complexes of $G$. Like the Zariski topology in algebraic geometry, this
topology is non-Hausdorff, if the maximal dimension $q$ of $G$ is positive. Indeed, 
if $v,w$ are zero-dimensional points contained in a simplex $x$, then $x \subset U(v) \cap U(w)$
so that $v$ and $w$ can not be separated by open sets in $\mathcal{O}$. 
Everything in this article is finite. Geometric realizations of finite abstract 
simplicial complexes \cite{DehnHeegaard} only appear for visualization purposes. 
No infinity axiom is ever involved.
\index{abstract simplicial complex}
\index{finite abstract simplicial complex}
\index{simplicial complex}
\index{open set}
\index{closed set}
\index{topology on simplicial complex}

\paragraph{}
The {\bf unit sphere} of $x \in G$ is defined to be the topological boundary 
$S(x)=\delta U(x) = \overline{U(x)} \setminus U(x)$ of the {\bf star} 
$U(x) = \{ y \in G, x \subset y \}$ of $x$. The unit sphere is a sub-simplicial complex because
it is the intersection of the {\bf closed ball} $B(x)=\overline{U(x)}$ and the complement
of the {\bf open ball} $U(x)$. The empty set $0=\emptyset=\{ \}$ is a simplicial complex by definition.
It is called the {\bf void}. If $x$ is a maximal simplex in $G$, then $U(x) = \{x\}$ and
$\overline{U(x)} = \{ y \in G, y \subset x \}$ is the {\bf core} of $x$.
The boundary $\delta U(x)$ is in that case the {\bf boundary complex} 
$\{y \in G, y \subset x, y \neq x\}$ of the simplex $x$ which is a 
closed set, a sub-simplicial complex of $G$. 
\index{Unit sphere}
\index{Unit ball}
\index{open ball}
\index{core}
\index{void}
\index{star}

\paragraph{}
The {\bf void}, the empty complex $0$ is 
the {\bf zero} $0$ in the {\bf monoid} of all simplicial 
complexes with {\bf join addition} 
$A \oplus B  = A \cup B \cup \{ a \cup b, a \in A, b \in B\}$, where $\cup$ 
is {\bf disjoint union}.  The {\bf Euler characteristic} of an arbitrary subset 
$A \subset G$ is defined as $\chi(A)=\sum_{x \in A} \omega(x)$, where 
$\omega(x)=(-1)^{\rm dim}(x)$ and ${\rm dim}(x)=|x|-1$. Here, $|x|$ denotes
the cardinality of $x \in X$. The maximum ${\rm max}_{x \in G} {\rm dim}(x)$ is the 
{\bf maximal dimension} of $G$ and denoted by $q$. 
\index{maximal dimension}
\index{join addition}
\index{zero}
\index{monoid}

\paragraph{}
Elements in $G$ of dimension $k$ are the sets of cardinality $k+1$ and are 
called {\bf $k$-simplices} or {\bf $k$-faces}, the $0$-simplices are also called
{\bf vertices}, the $q$-simplices {\bf facets}, 
the $1$-simplices {\bf edges}, the $(q-1)$ simplices 
{\bf walls}, the $2$-simplices are {\bf triangles} and the $(q-2)$-simplices
{\bf bones}.  The Euler characteristic is a {\bf valuation}, a map from subsets of $G$ to the integers
so that $\chi(A \cup B) = \chi(A) + \chi(B) - \chi(A \cap B)$ for all $A,B \subset G$.
As we review below, Euler characteristic is the only valuation that is invariant under Barycentric
refinement and for which $\chi(1)=1$ for the {\bf one point complex} $1=\{ \{ 1 \} \}$.
\index{walls}
\index{bones}
\index{valuation}
\index{vertex}
\index{edge}
\index{triangles}
\index{facet}
\index{face}
\index{simplex}

\paragraph{}
Inductively, a complex $G$ is called a {\bf Dehn-Sommerville $q$-sphere} 
if all unit spheres $S(x)$ are Dehn-Sommerville $(q-1)$-spheres 
and the Euler-gem formula $\chi(G)=1+(-1)^q$ holds. 
The {\bf zero complex} = empty complex=void $0=\{\}$ is the Dehn-Sommerville $(-1)$ sphere. 
Inductively, a complex $G$ is called a {\bf Dehn-Sommerville 
q-manifold} if every unit sphere $S(x)$ is a Dehn-Sommerville 
$(q-1)$-sphere. A Dehn-Sommerville $0$-manifold is a non-empty finite set of
vertices. The set of Dehn-Sommerville manifolds form, together with the $(-1)$ manifold $0$,
a sub-monoid, a monoid isomorphic to $\mathbb{N}=\{0,1,2,3, \dots \}$. 
In dimension $1$, every Dehn-Sommerville $1$-manifold is also a Dehn-Sommerville $1$-sphere, 
a finite union of circles, where a single circle is a cyclic complex $C_n$ 
with $n \geq 3$. For $n \geq 4$ it is the {\bf Whitney complex} = {\bf flag complex} 
= {\bf order complex} of the graph $C_n$, for
$n=3$, this is not the same than  $K_3=\{ \{1\},\{2\},\{3\},
\{1,2\},\{2,3\},\{1,3\}, \{1,2,3\} \}$. It is its $1$-skeleton complex
$C_3=\{ \{1\},\{2\},\{3\}, \{1,2\},\{2,3\},\{1,3\} \}$. To summarize this section, all 
1-manifolds are Dehn-Sommerville manifolds and Dehn-Sommerville spheres. This generalizes
to all odd dimensions $q$:
\index{Dehn-Sommerville sphere}
\index{Dehn-Sommerville manifold}
\index{zero complex}

\begin{thm}
Every odd dimensional Dehn-Sommerville $q$-manifold has zero Euler characteristic.
Therefore, every odd-dimensional Dehn-Sommerville manifold is also a Dehn-Sommerville sphere. 
\end{thm}

\begin{proof} 
Every complex can be partitioned into two sets $G=G_e+G_o$, where 
$G_e=\{ x \in G, {\rm dim}(x) \; {\rm even} \}$
and $G_o = \{ x \in G, {\rm dim}(x) \; {\rm odd} \}$. Euler characteristic is $\chi(G)=|G_e|-|G_o|$. 
The {\bf sphere formula} \cite{Sphereformula} is 
$\sum_x \omega(x) \chi(S(x)) = \sum_{x \in G_{e}} \chi(S(x)) - \sum_{x \in G_{o}} \chi(S(x))$. 
For Dehn-Sommerville q-manifolds, the unit spheres all have the same Euler characteristic. It is 
a constant in $\{0,2\}$. By the sphere formula, an odd-dimensional Dehn-Sommerville $q$-manifold 
must have zero Euler characteristic and so by definition is a Dehn-Sommerville sphere. 
\end{proof}

\paragraph{}
Because every odd-dimensional Dehn-Sommerville manifold is a Dehn-Sommerville sphere, 
the suspension of any odd dimensional Dehn-Sommerville $q$-manifold $M$ is an even dimensional 
$(q+1)$-Dehn-Sommerville manifold $G$. If $M$ is an even dimensional Dehn-Sommerville manifold that
is not a sphere, it sill can happen that the suspension is Dehn Sommerville. 
An example is when $M$ is a disjoint union of two projective planes. 
This is a Dehn-Sommerville $2$-sphere, because it is a 2-manifold with Euler characteristic $2$. 
Its suspension therefore is a Dehn-Sommerville $3$-manifold that is not a $3$-manifold. 
It is even a Dehn-Sommerville sphere, because every Dehn-Sommerville manifold 
is also a Dehn-Sommerville sphere. 

\section{A class of examples}

\paragraph{}
A {\bf $q$-manifold} is a simplicial complex for which we require that 
every unit sphere $S(x)$ is a {\bf $(q-1)$-sphere}.
A {\bf q-sphere} $G$ is a $q$-manifold such that there exists $x \in G$ such that 
$G \setminus U(x)$ and $U(x)$ are both contractible. A simplicial complex $G$ is called 
{\bf contractible}, if it is either the 1-point complex $1=\{ \{1\} \}=K_1$ 
or then if there exists a $x$ such that both the star 
$U(x)$ and the sub-complex $G \setminus U(x)$ are contractible. 
Since $q$-spheres have Euler characteristic $1+(-1)^q$ by the 
{\bf Euler-Gem formula}, every q-sphere is also a Dehn-Sommerville $q$-sphere 
and every $q$-manifold is also a Dehn-Sommerville $q$-manifold. 
\index{manifold}
\index{contractible}

\paragraph{}
The proof that a $q$-sphere $G$ satisfies the {\bf Euler-Gem formula} follows 
directly from the {\bf valuation formula}
$\chi(A \cup B) = \chi(A) + \chi(B) - \chi(A \cap B)$ 
applied to the case, where $A=G \setminus U(x)$ and 
$B=\overline{U(x)}$; in this case, $A \cap B = \overline{U(x)} \setminus U(x) = S(x)$.
The contractibility of $A$ and $B$ implies
$\chi(B)=1$ and $\chi(A)=1$. The induction assumption then gives
$\chi(A \cap B)=(1+(-1)^{q-1})$ so that we can conclude using the valuation formula
that $\chi(A \cup B) = 1+1-(1+(-1)^{q-1}) = 1+(-1)^q$. 
\index{Euler-Gem formula}
\index{valuation formula}
\index{contractibility}

\paragraph{}
An example of Dehn-Sommerville $2$-sphere $G$ that is not a
$2$-sphere is the union of $k$ disjoint $2$-spheres identified 
along corresponding north and south poles. Such a complex can 
be seen in various ways: $G$ is the suspension of a union of $k$ circles.
It is a {\bf branched cover} of a $2$-sphere with two branch points.
As it carries an action of the group $\mathbb{Z}_k$ with two fixed points, 
and because the quotient $M/T$ is the $2$-sphere, it is a {\bf ramified cover} of a 
$2$-sphere. It is a Dehn-Sommerville 2-sphere. 
More generally, the suspension of any odd-dimensional Dehn-Sommerville $q$-manifold 
is a Dehn-Sommerville $(q+1)$-manifold and also a Dehn-Sommerville $q$-sphere. 

\paragraph{}
We have seen that every odd-dimensional Dehn-Sommerville manifold is also a 
Dehn-Sommerville sphere. One can use the join operation to produce larger
odd-dimensional Dehn-Sommerville manifolds. This will be discussed in a section below. 
The following construction is analog to a construction 
for manifolds, where one picks a $q$-manifold for each vertex and
uses a {\bf connected sum construction} for each edge. The connected sum 
construction will be considered in a separate section below. 

\paragraph{}
So, here is a construction of odd-dimensional Dehn-Sommerville manifolds that are not manifolds.
Assume $q$ is even. Pick an arbitrary finite simple 
graph $\Gamma=(V,E)$ and replace every edge $e =(a,b)  \in E$ 
with a suspension $M_e \oplus \{a,b\}$ of a $(q-1)$-dimensional Dehn-Sommerville 
sphere $M_e$. Then replace every vertex $v$ with with neighbors $w_1, \dots, w_{d(v)}$
with a Dehn-Sommerville $q$-manifold $M_v$ in which $d(v)$ different points are chosen
to connect it with the neighboring points of $v$ in the graph. 
Alternatively, assume $M_v$ is a single point and glue all the edge manifolds at this point.
In both cases, we get a new larger dimensional Dehn-Sommerville $q$-manifold. 

\begin{figure}[!htpb]
\scalebox{0.75}{\includegraphics{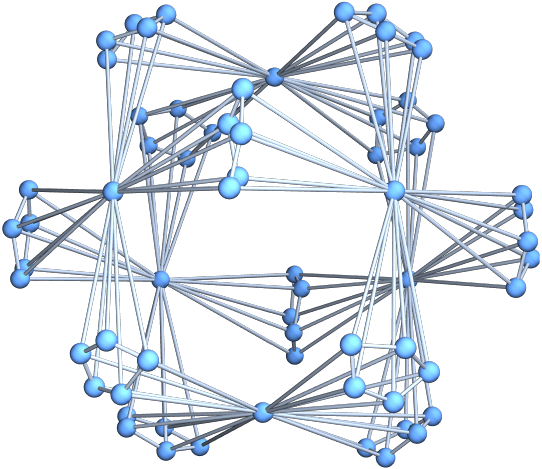}}
\label{Example1}
\caption{
A Dehn-Sommerville 2-manifold $G=M_{\Gamma,M}$ with 
$\Gamma=K_{2,2,2}$ and $M=C_5$ that is not a manifold. 
Its $1$-skeleton graph is a non-Hamiltonian graph. 
We have $b(G)=(1,7,12)$ and $\chi(G)=6+|E| b_1=6$.  
}
\end{figure}

\paragraph{}
The simplicial cohomology of $G_{\Gamma,M}$ can be computed explicitly 
from the cohomology of $M$ and the combinatorial data of the $1$-dimensional 
complex $\Gamma$.
\index{simplicial cohomology}

\begin{thm}
Given $\Gamma=(V,E)$ and an odd-dimensional Dehn-Sommerville manifold $M$
with $b(M)=(1,b_1,b_2, \dots, b_q)$, Then $G_{\Gamma,M}$ has Betti vector
$(1,|E|-|V|+1,|E| b_2, \dots,|E| b_q)$ and $\chi(M_{\Gamma,M} = |V|+|E| b_1$. 
\end{thm}

\begin{proof} 
The simplicial cohomology of the $1$-skeleton complex is $(1,|E|-|V|+1)$.
The simplicial cohomology of $G_{\Gamma,M}$  is $(1,|E|-|V|+1,|E| b_2, \dots |E| b_q)$.
Since $\chi(M)=0$, the Euler characteristic of the new complex is
$\chi(G_{\Gamma,M})=|V|+|E| \chi(M) = |V|+|E| b_1$. 
\end{proof} 

\paragraph{}
This shows that if we have an odd-dimensional Dehn-Sommerville manifold 
with a given $b_1$ and $n$ vertices, we can construct an even dimensional
Dehn-Sommerville manifold with $m*n$ vertices and $\chi(M) = m+m(m-1) b_1/2$.

\paragraph{}
Let us look at some examples of invariants $Y_{k,q}$,
the $(k+1)$'th eigenvector of $A_q$ understood as the
valuation $Y_{k,q}(G) = Y_{k,q}.f_G$.  \\

{\bf Example 1:} The {\bf K3 surface $G$} is one of most famous 4-manifolds. 
It is complex 2-manifold and a Calabi-Yau surface.
It can be seen as a quartic surface $x^4+y^4+z^4+w^4=0$, a 
(real co-dimension 2) sub-manifold in the complex $3$-manifold $\mathbb{CP}^3$ 
(a complex 3-manifold is a real $6$-manifold).
There is a rather small implementation with $f$-vector 
$f = (16, 120, 560, 720, 288)$. The dot product with the Barycentric valuation 
$Y_{2,4}=(0, -22, 33, -40, 45)$ is zero. 
The K3 surface satisfies Poincar\'e duality because it is orientable. Indeed, 
its Betti vector is $b=(1,0,22,0,1)$). \\
\index{K3 surface}

{\bf Example 2:} An other nice 4-manifold is the {\bf real projective space} $\mathbb{RP}^4$.
It is a 4-manifold that can carry positive curvature. There is a small 
implementation with $f = (16, 120, 330, 375, 150)$. Again, this is perpendicular 
to the Barycentric Dehn-Sommerville eigenvector  \\
$Y_{2,4} =(0, -22, 33, -40, 45)$. 
The $\mathbb{R}P^4$ manifold is not orientable and does not satisfy 
Poincar\'e duality. Its Betti vector is $b= (1,0,0,0,0)$. \\
\index{real projective space}
\index{projective space}

{\bf Example 3:} Let G be a {\bf twisted circle bundle} over the $3$-sphere. 
While the Betti vector of the product $S^3 \times S^1$ is $b=(1,1,0,1,1)$ 
and satisfies Poincar\'e duality, the twisted version does have the Betti vector 
$b=(1,1,0,0,0)$. It is not orientable. 
There is a small implementation with f-vector $f=(12, 60, 120, 120, 48)$. 
Again, $f.(0,-22,33,-40,45) = 0$.  \\
\index{twisted circle bundle}

{\bf Example 4:}  Let $G$ be the {\bf complex projective plane} $\mathbb{C}P^2$. 
It has Betti vector $b=(1,0,1,0,1)$. There is a small implementation with 
f-vector $f=(9, 36, 84, 90, 36)$. It satisfies the Dehn-Sommerville 
symmetry $f \cdot Y_{2,4} = f.(0,-22,33,-40,45) = 0$. \\
\index{complex projective space}

{\bf Example 5:} Lets take an example of a suspension $G$ of the 3-manifold $M=\mathbb{RP}^3$. 
A suspension of a 3-manifold $M$ is only a manifold if $M$ is a sphere. 
Since $M=\mathbb{R}P^3$ is not a sphere, $G=M \oplus 2=M \oplus S^0$ is no more a manifold. 
But it is a Dehn-Sommerville 4-manifold. 
The Betti vector of G is the same than the Betti vector of the 4-sphere, but it is 
not a 4-sphere. There is a small implementation of $G$ with 
f-vector $(13, 73, 182, 200, 80)$. 
Of course, it satisfies the Dehn-Sommerville symmetry. \\
\index{suspension}

{\bf Example 6}: Start with the 4-manifold from example 4. 
Since it has Euler characteristic 2, it actually is a {\bf Dehn-Sommerville 4-sphere}. 
Now take the suspension again to get a 5-manifold. 
It is the double suspension of $\mathbb{RP}^3$. It can also be seen as the join of 
$\mathbb{RP}^3$ with the $2$-sphere $\mathbb{S}^2$. 
The $f$-vector of a small implementation of this 5-manifold is 
$f=(15, 99, 328, 564, 480, 160)$. 
As a 5-manifold, this f-vector is perpendicular to the Euler characteristic vector 
$(1,-1,1,-1,1,-1)$. 
Beside the other "trivial" Dehn Sommerville vector $(0,0,0,0,-1,3)$, 
we have the non-trivial Dehn-Sommerville symmetry 
$Y_{2,5}=(0,0,-19,38,-55,70)$ in dimension 5. And indeed, 
$(15, 99, 328, 564, 480, 160) \cdot (0, 0, -19, 38, -55, 70) = 0$. \\
\index{double suspension}

{\bf Example 7}: Lets look at the suspension $S^0 \oplus K^3$ of the $K3$ surface. 
It is a $5$-dimensional simplicial complex. 
There is a small implementation with f-vector $f=( 18, 152, 800, 1840, 1728, 576)$. 
Now, $(18, 152, 800, 1840, 1728, 576) \cdot (0, 0, -19, 38, -55, 70) = 0$. 
But this was just an accident. The suspension of the K3 surface is neither a 
5 manifold nor a Dehn-Sommerville 5 manifold.  
Indeed, the Euler characteristic of this 5-manifold is $\chi(G)=-22$ and not $\chi(G)=0$ 
as it should be for any Dehn-Sommerville 5-manifold. But the other Dehn-Sommerville 
symmetries still hold. \\

{\bf Example 8}: If we look at the join $G=T^2 \oplus S^2$ of the 2-torus $T^2$ 
and the 2-sphere $S^2$, we get a $5$-dimensional complex. A small implementation 
has the $f$-vector $f(G)=(14, 84, 264, 448, 384, 128)$. It is no more perpendicular to 
$Y_{3,5}=(0, 0, -19, 38, -55, 70)$. It did not help that $S^2$ was a 
Dehn-Sommerville sphere. This already fails in smaller dimension. 
The join $T^2 \oplus S^1$ of $T^2$ with $S^1$ is a 
$4$-manifold which has an f-vector not perpendicular to 
$Y_{2,4}=(0,-22,33,-40,45)$.  \\

{\bf Example 9}: Let us look at the double suspension $S^1 \oplus (S^2 + T^2)$ of 
a disjoint union of a $2$-torus $\mathbb{T}$ and a 2-sphere $\mathbb{S}^2$. 
This is a Dehn-Sommerville 4-manifold but not a 4-manifold. 
There is a small implementation with $f=(18, 96, 224, 240, 96)$. 
It is perpendicular to $Y_{2,4}=(0, -22, 33, -40, 45)$.  
% G=JoinAddition[Addition[T2,Whitney[CompleteGraph[{2,2,2}]]],Whitney[CompleteGraph[{2,2}]]]; FVector[G].Y[1,4] 

{\bf Example 10}: Lets look at the double suspension of a {\bf homology 3-sphere}. 
It is a Dehn-Sommerville 5-manifold but not a 5 manifold.  
This is an interesting example because its geometric realization is 
actually homeomorphic to the 5 sphere by the double suspension theorem. 
There is a small implementation of G with f-vector 
$f=(28, 254, 972, 1786, 1560, 520)$. This is perpendicular to the 
Barycentric Dehn-Sommerville vector $Y_{3,5}=(0, 0, -19, 38, -55, 70)$. 
The other invariant is the Euler characteristic (which is zero) and the trivial 
valuation $Y_{5,5} = (0,0,0,0,-2,6)$.
% G=JoinAddition[HomologySphere,Whitney[CompleteGraph[{2,2}]]]; FVector[G].Y[2,5]
\index{homology sphere}

\section{Level sets}

\paragraph{}
As in the continuum, where manifolds can be constructed as level surfaces on a 
classical manifold, we have see that in the discrete if $G$ is a $q$-manifold and 
$g:V(G) \to A_k=\{0,\dots, k\}$ is a function write $g(x) = \{ g(w), w \subset x, 
{\rm dim}(w)=0 \}$. First extend $g$ from $V(G) \subset G$ to $G$ by 
$g( \{x_0, \dots, x_j\} ) = \{ f(x_i), i=0, \cdots ,j\}$ so that $g(x) \subset g(y)$ if 
$x \subset y$. Then define 
$$  G_g = \{ x \in G, g(x) =A_k\}  \; . $$
The set $G_g$ is an open set. But we can look at its Barycentric refinement to get
a simplicial complex. We first developed this for $k=1$ in \cite{KnillSard}
then extended it to higher co-dimension in 
\cite{DiscreteAlgebraicSets,DiscreteAlgebraicSets2}. 
\index{level surface}

\paragraph{}
Just because $G_g$ is a priori only an open subset of $G$ so that we need to define what we mean
it to be a Dehn-Sommerville q-manifold. We say that an open set $U \subset $ is a 
{\bf Dehn-Sommerville q-manifold}, if the graph $(V=U,E)$ is a Dehn-Sommerville 
$q$-manifold graph, where $E=\{ (a,b), a \subset b$ or $b \subset a$. For example, 
if $G=K_{2,2,2}$ is the Whitney complex of an octahedron and $f=1$ on two opposite
poles and $f=-1$ on the equator, then $G_f=U$ consists of all edges and faces that
hit the poles (this is the collection of all triangles and all edges that are not on the 
equator. The graph $(V=U,E)$ is a disjoint union of two $C_4$ graphs. We see therefore
the manifold $G_g$ as a disjoint union of two circles. 

\paragraph{}
We first want to point out that in a Dehn-Sommerville manifold $G$, there is for every 
$x \in V$ a decomposition $S(x) = S^-(x) \oplus S^+(x)$, where the {\bf stable sphere} $S^-(x)$ consists
of all simplices strictly contained in $x$ and the {\bf unstable sphere} 
$S^+(x)$ consists of all simplices that
strictly contain $x$. Writing the unit sphere as a join of a stable and unstable sphere is a 
{\bf hyperbolic structure} on $G$ that is present for all simplicial complexes. It is just that in the
Dehn-Sommerville case, both factors are Dehn-Sommerville spheres.
From the symmetry of the simplex generating functions 
$f_{S(x)}(t)$ and $f_{S^-(x)}(t)$, 
and the relation $f_{S^+(x)} f_{S^-(x)}(t) = f_{S(x)}(t)$ follows that also 
that $S^+(x)$ is a Dehn-Sommerville sphere. 
\index{hyperbolic structure}
\index{stable sphere}
\index{unstable sphere}

\paragraph{}
The following result - telling that the 
level set $G_g$ of a function $g:V \to \mathbb{R}$ is always a $(q-k)$-manifold if not empty - is 
the main reason why we decided to look at Dehn-Sommerville manifolds
in a more general sense and not look at Dehn-Sommerville spheres.

\begin{thm}[Dehn-Sommverville Level Sets]
If $G$ is a Dehn-Sommerville $q$-manifold and 
$g: V(G) \to A_k$ is an arbitrary function, then
$G_g$ is a Dehn-Sommerville $(q-k)$-manifold, if it is not empty. 
\end{thm} 
\begin{proof} 
Let $x$ be a $m$-simplex in $G$ on which $f$ takes all values. This means 
$g(x) :=\{ g(w), w \in x \}=A_k$. The
stable sphere $S^-(x) = \{ y \subset x, y \neq x \}$ is always a $(m-1)$-sphere
as can be seen by induction with respect to dimension noting that for a
$0$-dimensional $x$ the set $S^-(x)=\{ \}$ is the $(-1)$-dimensional sphere. 
The simplices in $S^-(x)$ on which $f$ still 
reaches $A_k$ is by induction a $(q-1-k)$-manifold and 
because it is a submanifold of the simplex boundary $S^-(x)$,
it has to be a $(q-1-k)$-sphere. (Also by induction, one can check that every
level set in a boundary complex $S^-(x)$ of a simplex $x$ is a boundary complex
of dimension $k$ less.) 
Every unit sphere $S(x)$ in in $G$ is a 
$(q-1)$-sphere and the join of $S^-(x)$ with $S^+(x)=\{ y, x \subset y, 
x \neq y \}$.
The sphere $S^+_g(x)$ in $M_g$ is the same than $S^+(x)$ in $M$ because every
simplex $z$ in $M$ containing $x$ automatically has the property that $g(z)=A_k$. 
So, the unit sphere $S(x)$ in $G_g$ is the join of a $(m-1-k)$-sphere and 
the $(q-m-1)$-sphere and so a $(q-k-1)$-sphere. Having shown that every unit
sphere in $G_g$ is a $(q-k-1)$-sphere, we see that $M_g$ is a $(q-k)$-manifold.
\end{proof}
\index{level set in Dehn-Sommerville manifold}

\paragraph{}
There is an analog version for varieties. It is simpler,
because there is no need to invoke Euler characteristic.  Here is the definition:
inductively, the void $0=\{\}$ is called the {\bf $(-1)$ Dehn-Sommerville
variety} and a {\bf Dehn-Sommerville q-variety} is a simplicial complex
for which every unit sphere is a Dehn-Sommerville $(q-1)$-variety. 
\index{Dehn-Sommerville variety}
\index{variety}

\begin{thm}[Dehn-Sommerville Varieties]
If $G$ is a Dehn-Sommerville $q$-variety and 
$g: V(G) \to A_k$ is an arbitrary function, then
$G_g$ is a $(q-k)$-variety, if it is not empty.
\end{thm} 

\begin{proof}
The proof parallels the proof before. The key of the inductive proof
is again the {\bf hyperbolic structure}
which exists in any simplicial complex: 
the unit sphere $S(x) = \{ y \in G, x \neq y$ and $y \subset x  x \subset y \}$
is the join of $S^-(x)  = \{ y \in G, y \neq x, y \subset x \}$ 
and $S^+(x) = \{ y \in G, y \neq x, x \subset y \}$. 
If $x$ is a $m$-simplex, then the stable sphere $S^-(x)$ is 
the boundary sphere of the simplex $x$. It is 
a simplical complex and a $(m-1)$-sphere. As for the open set $S^+(x)$,
it is a Dehn-Sommerville $(q-k-1)$-variety as all unit spheres $S(y)$ are
Dehn-Sommerville $(q-k-2)$-varieties. Note that $S(y)$ for $y \in S^+(x)$ is
the same than $S^+(y)$ in $S^+(x)$. 
Now $S_{G_g}(x) = S^-(x) \oplus S^+_g(x)$ which by induction is the join of 
two varieties and so a Dehn-Sommerville variety. 
\end{proof}
\index{level set in variety}

\paragraph{}
An example of a Dehn-Sommerville 1-variety that is not a Dehn-Sommerville 1-manifold
is the {\bf star graph} $S_n$ or the {\bf cube graph} or the {\bf dodecahedron graph}.
Every unit sphere is either the $n$-point graph $n$ or the 1-point graph $1$. 
More generally, any 1-dimensional simplicial complex which is pure (meaning  in the
1-dimensional case that there are no isolated vertices)
is a Dehn-Sommerville 1-variety. 
Dehn-Sommerville varieties are a decent generalization of 
Dehn-Sommerville manifolds because they form a monoid. We get to this
in the next section. 

\section{Monoids}

\paragraph{}
The suspension $M \oplus S^0$, the join of $M$ with a $0$-sphere $S^0=\{ \{a\},\{b\} \}$ 
of a $q$-manifold $M$ is already not a manifold if $M$ is not a $q$-sphere. 
It is however a Dehn-Sommerville manifold, if $M$ is a Dehn-Sommerville sphere. 
The reason is that the two points $\{a\},\{b\}$ have as 
unit spheres the original manifold $M$ that is not a sphere. 
An example of a Dehn-Sommerville 3-manifold that is not a 3-manifold 
is the suspension of a disjoint copy of two projective planes 
$M=\mathbb{P}^1  + \mathbb{P}^1$ which has Euler characteristic $2$.
The suspension of any $3$-manifold is a Dehn-Sommerville 
$4$-manifold. The double suspension of a homology $3$-sphere is only a 
{\bf Dehn-Sommerville 5-manifold} and not a classical 5-manifold.
However, by the double suspension theorem it is in the continuum 
homeomorphic to a 5-sphere. But homeomorphisms in the continuum can be very complicated. 
Still we see from this example that there are Dehn-Sommerville 5-manifolds that are not
4-manifolds but which have geometric realizations that are topological continuum 5-spheres. 
\index{monoid}

\begin{thm}
The set of all Dehn-Sommerville spheres is a 
submonoid of all simplicial complexes.
\end{thm}

\begin{proof}
We use induction with respect to $q$ and the
fact that $S_{A \oplus B}(x) = S_A(x) \oplus B$ if $x \in A$ 
and $S_{A \oplus B}(x) = A \oplus S_B(x)$ if $x \in B$. Since $S_A(x)$ rsp.
$S_B(x)$ are both one dimension lower, the monoid property for dimension 
$q-1$ implies the property for dimension $q$. By Theorem~(\ref{1})
odd-dimensional Dehn-Sommerville manifolds are Dehn-Sommerville spheres. 
The induction foundation is $q=(-1)$, where both $A,B$ are $0$. \\
The compatibility with the Euler characteristic follows from 
$f_{A \oplus B}(t) = f_A(t) f_B(t)$ and $\chi(A)=1-f_A(-1)$ and the 
fact that $chi(A) \in \{0,2\}$ is equivalent to $f_A(-1) \in \{ -1,1\}$. 
\end{proof} 

\begin{coro}
The set of all odd dimensional Dehn-Sommerville manifolds is a
submonoid of all simplicial complexes.
\end{coro}

\begin{proof}
An odd dimensional Dehn-Sommerville manifold is a Dehn-Sommerville sphere.
\end{proof} 

\paragraph{}
While the focus is on Dehn-Sommerville manifolds, there is an analog for
the already mentioned {\bf Dehn-Sommerville varieties} which is the smallest
monoid that is invariant under the property that all unit spheres are in this
class.  The same formula for the unit sphere of the join imply that: 

\begin{thm}
The set of all Dehn-Sommerville varieties form a
submonoid of all simplicial complexes.
\end{thm}

\begin{comment} 
We could even look at the subclass of varieties $G$ for which the 
stable spheres $S^+(x)=\{ y \in G, x \subset y, x \neq y\}$ are
spheres for all $q$ and $q-1$ simplices $x$ in $G$.
In that case, we still have a notion of sectional curvature
and a geodesic flow \cite{geodesics1,geodesics2}. 
\end{comment} 

\paragraph{}
An other monoid is the set of all {\bf contractible complexes}, if we define
a complex $G$ to be {\bf contractible}, if it the one point complex or if there
exists $x$ such that both 
$S(x) = \delta U(x)$ and $G \setminus U(x)$ are contractible. It actually 
has the ideal property that the join of a 
non-zero contractible complex with an arbitrary complex is contractible. 
The smallest monoid containing all $0$-dimensional complexes
is the {\bf partition monoid} \cite{KnillCuisenaire} whose elements are
multi-partite graphs $K_{n_1, \dots, n_k} = n_1 \oplus n_2 \cdots n_k$ 
with $\sum_{j=1}^k n_j = q$. It contains complete complexes $K_n$, 
the utility graph complex $K_{3,3}$ or the octahedron $K_{2,2,2}$ or
$q$-dimensional spheres $K_{2,\dots, 2}$. 
\index{contractible complex}

\paragraph{}
One could consider much more examples:
given an arbitrary monoid $\mathcal{M}$ of complexes, one can look at all 
complexes that are homotopic to an element in such a monoid.  
One can for example look at the monoid that contains all complexes that contract to an
odd dimensional Dehn-Sommerville $q$-manifold. It contains most Vietoris-Rips complexes 
of classical continuum q-manifolds: take a random point cloud, remove successive 
points with contractible unit spheres, until only points remain for 
which all unit spheres are (q-1)-spheres. 
The monoid of {\bf Whitney complexes of graphs} is isomorphic to the monoid of graphs
and slightly smaller than the monoid of all simplicial complexes.
Finally, we should point out that {\bf delta sets} and more generally, 
{\bf abstract delta sets} given by a tripleet
$(G,D,R)$ (where D is a Dirac matrix and R a dimension funcion) 
form monoids containing the simplicial complex monoid. For the later
just note that given exterior derivative $d$ on $G$ and an exterior derivative on 
$H$ naturally defines an exterior derivative on $G \oplus H$. One can define 
{\bf Dehn-Sommerville abstract delta sets} in the same way by induction as 
$(G,D,R)$ defines an order structure
which defines a topology and so has a notion of unit sphere. 

\section{Invariants}

\paragraph{}
Given an arbitrary simplicial complex $G$. 
The number $f_k(G)$ counting the number of $k$-simplices in $G$ defines the
{\bf $f$-vector} $f=(f_0,f_1, \dots, f_q)$ and the {\bf $f$-generating function} 
$f_G(t) = 1+sum_{k=0}^q f_{k} t^k$ which is a polynomial of degree $q$.
An integer vector $(X_0, \dots, X_q)$ in $\mathbb{Z}^{q+1}$ defines 
a {\bf valuation} $X(G) = X \cdot f  = X_0 f_0 + \dots X_q f_q$. 
By distributing the values $X_k$ attached to each $k$-simplex in $G$ equally to its 
$k+1$ vertices, we get the {\bf curvature}
$K(v) = \sum_{k=0}^{q} X_k f_{k-1}(S(x))/(k+1)$ for the valuation $X$ and
$\sum_{v \in V} K(x) =X(G)$ is the {\bf Gauss-Bonnet theorem} for the 
valuation $X$. In the case $X=(0,1,0, \dots, 0)$ it leads to 
$X(G)=f_1(G)$, and curvature is $K(v)={\rm deg}(v)/2$ and Gauss-Bonnet is 
the ``Euler handshake". If $X=f_q$ is the volume of $G$, then 
$K(v)={\bf vol}(U(v))/(q+1)$, The valuation $X=(1,-1,1,-1,\dots)$ 
with $X_k=(-1)^k$ gives the {\bf Euler characteristic}.
\index{Euler handshake}
\index{Euler characteristic}

\paragraph{}
Of special interest are {\bf Dehn-Sommerville valuations}, defined by the vectors
$$ X_{k,q} = \sum_{j=k}^{q} (-1)^{j+q} \left( \begin{array}{c} j+1 \\ 
                                 k+1 \end{array} \right) e_j - e_k  \; , $$
where $e_0=(1,0,0, \dots, 0), \dots e_q=(0,0,\dots,1)$ are basis vectors. 
We assume $k<q$ because this is zero for $k=q$.
If the vectors $\left( X_{0,q}, \dots, X_{q-1,q} \right)$ are written as row vectors 
in a matrix $X_q$, we have
$$ X_1=\left[\begin{array}{cc} -2 & 2 \\ \end{array} \right],
   X_2=\left[\begin{array}{ccc} 0 & -2 & 3 \\ 
                                0 & -2 & 3 \\ \end{array} \right], 
   X_3=\left[\begin{array}{cccc} -2 & 2 & -3 & 4 \\ 
                                  0 & 0 & -3 & 6 \\ 
                                  0 & 0 & -2 & 4 \\ \end{array} \right],
   X_4=\left[ \begin{array}{ccccc} 0 & -2 & 3 & -4 & 5 \\
                                   0 & -2 & 3 & -6 & 10 \\
                                   0 &  0 & 0 & -4 & 10 \\
                                   0 &  0 & 0 & -2 & 5 \end{array} \right] \; . 
$$
\index{Dehn-Sommerville valuation}

\begin{comment}
% TeXForm[Dvectors[4]]  January 20, 2015
X[k_,q_]:=Module[{f},f[i_]:=Table[If[i==j-1,1,0],{j,q+1}];
  Sum[(-1)^(j+q) Binomial[j+1,k+1] f[j],{j,k,q}]- f[k]];
Table[X[k,3],{k,0,3}]
F[k_,q_]:=Delete[X[k,q],1]*Table[1/(l+1),{l,1,q}]==X[k-1,q-1]/(k+1)
Do[Print[F[k,8]],{k,0,8}]
\end{comment}

\begin{lemma}[Curvature lemma]
For any simplicial complex, the curvature of the valuation 
$X_{k,q}$ is $K(v) = X_{k-1,q-1}(S(v))$.
\end{lemma}
\begin{proof}
This is the combinatorial identity 
$$   X_{k,q}(j+1)/(j+1) = X_{k-1,q-1}(j)/(k+1)  \;   $$
which translates into
$$ (-1)^{j+q+1} B(j+2,k+1)/(j+1) = (-1)^{j+q+1} B(j+1,k)/(k+1) \; . $$
For $k=-1$ one has to define $B(x,y)=\Gamma(x + 1)/\Gamma(y + 1) \Gamma(x - y + 1)$. 
In the limit $k \to -1$ we need to look at $X_{k-1,q-1}(j)/(k+1)$ as a limit.  
\end{proof}
\index{Curvature lemma}

It follows that

\begin{thm}
Dehn-Sommerville valuations $X_{k,q}$ are all zero for 
Dehn-Sommerville spheres if $k+q$ is even.
\end{thm}

\begin{proof}
Use Gauss-Bonnet $\sum_{x \in G} K(x)$ and induction.
For $q=1$, a Dehn-Sommerville manifold is a union of 
cyclic complexes for which the Dehn-Sommerville valuation $X(0,1)$ 
The reason is that $X(-1,q) \cdot f(G)$ is $(-1)^q \chi(G)$. 
\end{proof} 

\section{Eigenvalue Functionals} 

\paragraph{}
The valuations $X_{k,q}$ invoked Binomial identities. An other approach to the
identities is via eigenfunctions of the Barycentric refinement operator. 
If $f=(f_0, \dots, f_q)$ is the {\bf $f$-vector},
and $f_G(t) = 1+\sum_{k=0}^d f_k(G) t^{k+1}$ 
is the {\bf simplex generating function}, define $F_G(t)=\int_0^t f_G(s) \; ds$. 
In \cite{dehnsommervillegaussbonnet} we noted
that Gauss-Bonnet can be written as a functional identity. 
Let $V \subset G$ denote the set of vertices, zero-dimensional simplices. We have

\begin{thm}[Functional Gauss-Bonnet]
For all simplicial complexes, $f_G(t) = 1+\sum_{v \in V} F_{S(v)}(t)$.
\end{thm}

\paragraph{}
It follows from this formula that if all $f_{S(v)}$ are odd with respect to $-1/2$
then $F_{S(v)}$ is even with respect to $-1/2$ so that $f_G$ is even with respect
to $-1/2$. On the other hand if $f_{S(v)}$ are all even with respect to $-1/2$ 
then $F_{S(v)}$ and so $f_G$ is odd with respect $-1/2$, provided that 
$f_{G}(-1)=-1$ meaning $\chi(G)=0$. 

\paragraph{}
Here is the proof of the functional Gauss-Bonnet formula:

\begin{proof}
As all simplex generating functions satisfy $f_G(0)=1$, 
the statement is equivalent to $f'_G(t) = \sum_{v \in V} f_{S(v)}(t)$ and gives for $t=-1$
the {\bf Levitt Gauss-Bonnet formula} for Euler characteristic. The proof is an observation:
attach to every $y \in S(x)$ an energy $t^{k+1}$ so that $f_G(t)$ is the total energy
Every $k$-simplex $y \in S(x)$ defines a $(k+1)$-simplex $z$ in $U(x) \subset G$
carrying the charge $t^{k+2}$. It contains $(k+2)$ vertices.
Distributing this charge equally to these points gives each a charge $t^{k+2}/(k+2)$.
The curvature $F_{S(x)}(t)$ adds up all the charges of $z$.
\end{proof} 

See \cite{cherngaussbonnet,Sphereformula,dehnsommervillegaussbonnet}. 
% text3/graphgeometry/art-characteristic/sphereformula

\paragraph{}
Given a finite simplicial complex $G$, the {\bf $h$-function} is defined as
$h_G(t) = (t-1)^q f_G(1/(t-1))$. It is a polynomial
$h_G(t) = h_0 + h_1 t + \cdots + h_q t^d + h_{q+1} t^{q+1}$,
defining a {\bf $h$-vector} $(h_0,h_1, \dots, h_{q+1})$. 
Let us say a complex satisfies {\bf Dehn-Sommerville functional relations} 
if $h$ is {\bf palindromic}, meaning that $h_i=h_{q+1-i}$ for all $i=0, \dots, q+1$. 
\index{palindromic}

\begin{lemma}
The following are equivalent: \\
a) $f_G(t)= (-1)^q f_G(-1-t)$. \\
b) $h$ is palindromic.  \\
c) $g(t) = f(t-1/2)$ is either {\bf even} or {\bf odd}. 
\end{lemma}

\begin{proof}
The palindromic condition of $h$-vector means that the roots of the $h$-function 
$h(t)=1+h_0 t + \cdots + h_q t^{q+1}  = (t-1)^q f(1/(t-1))$ is invariant under the 
involution $t \to -1-t$  which is equivalent to the symmetry $f(-1-t)= \pm f(t)$ 
and in turn means that $g(t)$ is either even or odd. 
\end{proof}

\paragraph{}
Examples:  for the 3-sphere $G$ with $f_G(t) = 1 + 8t + 24t^2  + 32t^3  + 16t^4$, 
we have $f_G(t-1/2)=16 t^4$ which is an even function. 
For the {\bf Barnette sphere} \cite{Barnette1973}, a $3$-sphere with 
$f$ vector $(8, 27, 38, 19)$ and $f$-function $f(t) = 1 + 8 t + 27 t^2  + 38 t^3  
+ 19 t^4$, we have $f(t-1/2) = 3/16 - 3t^2/2 + 19 t^4$, 
which is an even function. 
\index{Barnette sphere}

\paragraph{}
The unitary involutive map $T(f)(x) = f(-1-x)$ on the linear space of polynomials of degree $q$
is a reflection on $\mathbb{R}^{q+1}$ and has eigenvalues $1$ and $-1$. The spectral theorem of 
normal operators defines an eigenbasis. 
In analogy to $\tilde{T}(f)(x)=f(-x)$, call eigenfunctions of $1$ of $T$
{\bf even functions} and eigenfunctions of $-1$ {\bf odd functions}.
As $T$ and the Barycentric operation $A$ commute, they have the same eigenbasis.
The eigenbasis of $A$ is also an eigenbasis of $T$: even eigenvectors are
eigenfunctions of $T$ to tTohe eigenvalue $1$ and odd eigenvectors of $A$
are eigenfunctions of $T$ to the eigenvalue $-1$. We have verified: 

\begin{thm}
For even $q$, the even eigenvectors of $A^T$ and for odd $d$, 
the odd eigenvectors of $A^T$ define functionals which are zero 
for all Dehn-Sommerville spheres. 
\end{thm}

This is covered in Theorem~(1) in \cite{valuation} and \cite{dehnsommervillegaussbonnet}.
More details are given also below. 

\section{Refinements}

\paragraph{}
The {\bf Barycentric refinement} $G_1$ of $G$ is the order complex of $G$.
It is the Whitney complex of the graph $\Gamma(G)=(V=G,E)$, where
$E$ is the set of pairs $(x,y)$ with $x \subset y$ or $y \subset x$.
The Barycentric refinement operation and the Dehn-Sommerville symmetry
are compatible. This is not a surprise, given that one of the descriptions 
uses the eigenfunctions of the Barycentric refinement operator. 
 
\begin{thm}
If $G$ is Dehn-Sommerville then $G_1$ is Dehn-Sommerville.
\end{thm}
\begin{proof}
We use induction with respect to dimension.  
The unit spheres of $G$ are either (q-1)-spheres, boundary spheres of
simplices or Barycentric refinements of a unit sphere of $G$. 
Both cases satisfy Dehn-Sommerville by induction.
\end{proof}

\paragraph{}
The {\bf Barycentric refinement operator}  $A$ is defined
as $f(G_1)=A f(G)$. If $G$ has maximal dimension $q$,
then the matrix $A$ is a $(q+1) \times (q+1)$ upper triangular
$A_{ij} = {\rm Stirling}(i,j) i!$, involving the 
{\bf Stirling numbers of the second kind}.
All eigenvalues $\lambda_k=k!$ are distinct. 
\index{Stirling numbers}
\index{Barycentric refinement operator}

\paragraph{}
The linear unitary reflection involution $T(f)(x) = f(-1-x)$ on polynomials 
defines an involution on $\mathbb{R}^{d+1}$. 
By the spectral theorem for normal operators, there is an
eigenbasis for the reflection. 
Call eigenfunctions of $1$ {\bf even functions} and 
eigenfunctions of $-1$ {\bf odd functions}. 
\index{even eigenfunction}
\index{odd  eigenfunction}

\paragraph{}
Any eigenvector $X$ of $A^T$ defines a valuation
$$   \phi_X G = X \cdot f_G \; . $$
If $\lambda$ is an eigenvalue, then 
$\phi_X G_1 = \lambda \phi_X G$ because
$\phi_X G_1 = X \cdot f_{G_1} = X \cdot A f_G = A^T X \cdot f_G = \lambda X \cdot f_G$. 

\paragraph{}
Most $\phi_X(G_n)$ explode in general under refinements $G_n$.
The only functional that stays invariant is  $\phi_{\chi}$, the 
Euler characteristic.

\begin{lemma}
The eigenbasis of $A$ is also an eigenbasis of $T$.
\end{lemma}
\begin{proof}
The linear operators $T$ and the Barycentric operation $A$ commute. 
and so share an eigenbasis. 
\end{proof} 

\begin{coro}
For even $d$, the even eigenvectors of $A^T$ and
for odd $d$, the odd eigenvectors of $A^T$ define functionals 
which are zero on Dehn-Sommerville spaces.
\end{coro} 

\paragraph{}
This was Theorem~(1) in \cite{valuation},
where the idea of proving Dehn-Sommerville via curvature has hatched
and multi-variate versions of Dehn-Sommerville answered 
a question of Gruenbaum \cite{Gruenbaum1970} from 1970. 
In multi-dimensions, the Dehn-Sommerville symmetry
just has to hold for each of the variables appearing in the simplex generating function 
$f(t_1, \cdots ,f_m)$. The proof in higher dimensions is identical using Gauss-Bonnet. 

\paragraph{}
Let $Y_{k,q}$ denote the $k+1$'th eigenvector of $A_q$. It defines a
functional $Y_{k,q}(G) = Y_{k,q}.f_G$. By definition 
$Y_{k,q}(G_1) = k! Y_{k,q}(G)$.  

\paragraph{}
Every valuation is a linear combination of eigenvectors of $A_q$. 
We have seen that there is a $[q/2]$ dimensional space of valuations
which are zero on Dehn-Sommerville spaces. Assume $X=\sum_i a_i Y_i$
is such an invariant, then $X(G_1) = \sum_i a_i i! Y_i$ implying
that each valuation that is zero must be a linear combination of 
eigenvectors for which each is a valuation that is zero. This means
that there are $[q/2]$ eigenvectors which are Dehn-Sommerville invariants. 

\begin{thm}
The span of $\{ X_{k,q}, 0 \leq k \leq q, k+q  \; {\rm odd}  \}$ and
$\{ Y_{k,q}, 0 \leq k \leq q, k+q  \; {\rm odd}  \}$ are the same. 
\end{thm}

\paragraph{}
While we have an equivalence between the functionals 
$X_{k,q}$ and $Y_{k,q}$, we do not know yet how to make this explicit. For example,
$$ Y_{1,6} = a X_{0,6} + b X_{2,6} + c X_{4,6}  $$ 
with $a = 3949, b = -359, c = 169$. 

\paragraph{}
To summarize: there are four different ways to see the Dehn-Sommerville symmetries:

\begin{thm}[Dehn-Sommerville symmetries]
Let $G$ be a Dehn-Sommerville $q$-sphere (defined using unit spheres). Then  \\
a)  The valuations $X_{k,q}$ are zero if $k+q$ is even (Gauss-Bonnet). \\
b)  The valuations $Y_{k,q}$ are zero if $k+q$ is even (Barycentric refinement)   \\
c)  $f_G(t)= (-1)^q f_G(-1-t)$, meaning $f_G$ is either even or odd (functional Gauss-Bonnet) \\
d)  $h_G(t)$ is palindromic (change of coordinates).
\end{thm}

\section{Chromatic number} 

\paragraph{}
In this section we generalize the chromatic number estimate $\chi(G) \leq 2q+2$ from 
q-manifolds to Dehn-Sommerville q-manifolds. First of all, we have to say what the 
chromatic number of a simplicial complex is as it is traditionally defined for graphs. 
But any simplicial complex $G$ defines a 1-dimensional simplicial complex, the 
1-skeleton complex in which the $0$-simplices are the vertices and the $1$-simplices are
the edges. A {\bf coloring} is an assignment of values to the vertices $V = {\rm dim}^{-1}(0)$ 
of the complex $G$. The {\bf chromatic number} is then the minimal target set which allows
to color $V$ in such a way that adjacent vertices have different colors. For a Dehn-Sommerville
q-manifold of course the chromatic number is bound below by $q+1$ as we have $q$-dimensional 
simplices there which consist of $q+1$ vertices all connected to each other. 
By using a theorem of Whitney one can see that the {\bf 4-color theorem} is equivalent to the
statement that all $2$-spheres have chromatic number 4 or less. Note that this does not 
work for Dehn-Sommerville 2-spheres, as the disjoint union of two projective planes is a 
Dehn-Sommerville 2-sphere and that there are projective planes with chromatic number 5. 
We show here that for all Dehn-Sommerville 2-manifolds, the chromatic number is maximally 6
and more generally that for all Dehn-Sommerville q-manifolds, the chromatic number is maximally $2q+2$. 
We have no example yet, for which the chromatic number of a 2-manifold is actually 6. For 2-manifolds, 
a conjecture of Albertson-Stromquist \cite{AlbertsonStromquist} states that the chromatic number
is 5 or less. Any Barycentric refinement of a Dehn-Sommerville q-manifold of course has
chromatic number $q+1$ as this holds for arbitrary simplicial complexes: the dimension functional 
is a coloring. 
\index{chromatic number}
\index{Albertson-Stromquist conjecture}
\index{4-color theorem}

\paragraph{}
Let $\Gamma=(V,E)$ be a finite simple graph. 
The {\bf vertex arboricity} ${\rm ver}(\Gamma)$ \cite{ChartrandKronkWall1968} 
is the maximal number of forests partitioning $V$ such that each 
forest generates itself in $\Gamma$. It originally was also called {\bf point arboricity}.
The chromatic number ${\rm chr}(\Gamma)$ of a graph $\Gamma$ 
is the maximal number of colors that can be assigned to $V$ such that neighboring vertices have different colors.
Vertex arboricity sandwiches the chromatic number ${\rm chr}(\Gamma)$ with
${\rm ver}(\Gamma) \leq {\rm chr}(\Gamma) \leq 2 {\rm ver}(\Gamma)$. Despite name appearance,
there is a severe difference between vertex and the {\bf edge arboricity} ${\rm arb}(\Gamma)$
(the minimal number of forests partitioning the edge set of $\Gamma$).
Computing ${\rm ver}(\Gamma)$ is a NP-hard, while 
${\rm arb}(\Gamma)$ is a polynomial task by the {\bf Nash-Williams theorem}.
\index{vertex arboricity}
\index{point arboricity}
\index{edge arboricity}
\index{Nash-Williams theorem}

\paragraph{}
The {\bf dual $\Gamma=\hat{G}$} of a Dehn-Sommerville q-manifold $G$ has the facets of $G$ as 
vertices and connects two if they intersect in a $(q-1)$-simplex. For example, the dual graph of the 
octahedron complex $K_{2,2,2}$ is the cube graph, the dual graph of the icosahedon complex is the 
dodecahedron graph.
\index{dual graph} 

\begin{lemma}
If $G$ is a Dehn-Sommerville q-manifold then $\hat{G}$ is a $(q+1)$-regular triangle-free graph. 
\end{lemma}

\begin{proof} 
For every $k$-simplex $x=(x_0,\dots,x_{k})$ in a q-manifold, 
the intersection of all unit spheres $S(x_j)$ is a 
$(q-k-1)$-dimensional sphere. This means that for $k=q-1$
the intersection is a 0-Dehn-Sommerville sphere
which is a 2 point graph. This shows that any of the walls can not be
part of more than two maximal simplices meaning that $\hat{G}$ is 
triangle free. If there is no boundary, then every facet $x$ is 
surrounded by $q+1$ other facets, as there are $q+1$ walls in q q-simplex 
$x$. 
\end{proof} 
\index{triangle free graph}

\paragraph{}
The vertex arboricity of triangle free graphs can be arbitrarily large,
even on the class of triangle-free graphs. The reason is that the 
chromatic number can be arbitrarily large \cite{Mycielski}. 
This is unlike in the planar case where the chromatic number
is $\leq 3$ by {\bf Groetsch's theorem}. Dual graphs of manifolds are of a 
special kind however. Unlike the Mycielski examples, their vertex arboricity
is simple:
\index{Mycielski example}
\index{Groetsch theorem}
\index{planar graph}

\begin{thm}
The vertex arboricity of the dual $\hat{G}$ any Dehn-Sommerville q-manifold $G$ with
or without boundary is always 2 if $q>0$. 
\end{thm}

\paragraph{}
This results from the following proposition. We can cut a manifold $G$ to get two
manifolds $K,H$ with common boundary $C$. By induction in $q$, the vertex arboricity of $C$
is $2$. The proposition in turn then allows to extend the forests to the interiors
and so to the entire manifold $G$. Any forest pair on the boundary of a manifold can
grow into the interior of the manifold. First of all, the {\bf boundary} of the dual
$\hat{G}$ of a q-manifold $G$ consists of the vertices $x$ (facets in $G$) for which there
is one or more walls with only one facet $x$ attached. 

\begin{propo}[Growing forests into the interior]
If $G$ is a Dehn-Sommerville $q$-manifold with boundary and assume the boundary 
of $\hat{G}$ has been covered with two forests generating themselves, then the 
two forests can be extended into the interior of $\hat{G}$, still 
generating themselves.
\end{propo} 

\begin{proof}
Take a minimal counter example $G$ (meaning a manifold with minimal $q$ and for this $q$ with minimal number of vertices
in $\hat{G}$) can not be extended to the interior. As it can not be a ball, by the 
next lemma, we an cut it into two manifolds with boundary $G \setminus U(x),B(x)$, where $x$
is a wall on the boundary. As the original $G$ was minimal, the proposition applies to $G \setminus U(x)$.
If we choose forests in both boundaries, agreeing on their original intersection face, we can extend the 
forests to the interior of each of the parts. Now glue the two parts together. This covers $G$ with two forests
that both generating themselves. This contradicts the assumption that $G$ lacked this property.
\end{proof} 

\paragraph{}
Note that a manifold with boundary does not have to have interior zero dimensional points. 
For example, a single q-simplex is a q-manifold $G$ with boundary. Its dual $\hat{G}$ is a single
point. 

\begin{lemma}
Take $G=B(x)$ with boundary $S(x)$. If $\hat{S(x)}$ is covered with 2 trees, this can be extended to
the interior $\hat{B(x)}$.
\end{lemma}
\begin{proof}
Assign the center of the ball to either tree color. In both cases, we have
two trees covering $\hat{G}$.
\end{proof}

\begin{lemma}
Given a Dehn-Sommerville q-manifold $G$ with boundary and $x$ is a wall at the boundary 
then $G \setminus U(x)$ is a Dehn-Sommerville $q$-manifold. 
\end{lemma}

\begin{coro}
The chromatic number of a Dehn-Sommerville manifold is $2q+2$ or less.
\end{coro}
\begin{proof}
Since the vertex arboricity of the dual $\hat{G}$ is $2$, we can partition
the vertices of $\hat{G}$ into two forests $A,B$. Use $q+1$ colors to color
each facet in $G$ that belongs to a vertex $A \subset V(\hat{G})$ and an other set
of $q+1$ colors to color each facet belonging to $B$. 
\end{proof} 

\paragraph{}
We had prove in \cite{HamiltonianManifolds} that manifolds are Hamiltonian. This
property however does not extend to Dehn-Sommerville manifolds. 
The complex defined in Figure 1 is a Dehn-Sommerville 2-manifold that is not 
Hamiltonian, meaning that the skeleton graph is not a Hamiltonian graph. 

\section{Constructions}

\paragraph{}
The {\bf Cartesian product} of two simplicial complexes $A,B$ is the 
Whitney complex of the graph $(V,E)$ in which $V=\{ (a,b) \; a \in A, b \in B \}$
are the vertices and where two different
vertices $(x,y),(u,v)$ are connected by an edge if either $x \subset u, y \subset v$ or
$u \subset x, v \subset y$.  This product is also known as the {\bf Stanley-Reisner
product}. We already have seen that the Stanley-Reisner product of a 
$p$-manifold with a $q$-manifold is a $(p+q)$-manifold.
\index{Cartesian product}
\index{Stanley-Reisner product}

\begin{thm}
If $G$ is a Dehn-Sommerville $p$-manifold and $H$ is a Dehn-Sommerville $q$-manifold, 
then $G \times H$ is Dehn-Sommerville $p+q$-manifold. 
\end{thm}
\begin{proof} 
Denote by $B_H(y)$ the unit ball of a vertex $y \in H$. Its boundary is the unit sphere $S_H(y)$
in $H$. The unit sphere $S_{G \times H}(x,y)$ is the union of two generalized cylinders
$S_G(x) \times B_H(y)$ and $B_G(x) \times S_H(y)$ glued along 
$S_G(x) \times S_H(y)$. One can see by induction that this is a 
Dehn-Sommerville $(q-1)$-sphere: if $z$ is a point in this space
then its unit sphere is again of this form but of dimension $1$ lower. 
\end{proof}

\paragraph{}
The {\bf connected sum} of two simplicial complexes $A,B$ of the same maximal dimension $q$ 
is obtained by removing two isomorphic unit balls $U(x),U(y)$ and identifying along $S(x),S(y)$. 
This construction does not need $A,B$ to be manifolds. It works for two general manifolds as long
as the maximal dimension is the same. There are always ways to do that. 
A simple example is to take $x \in A, y \in B$ of maximal
dimension so that $S(x),S(y)$ are isomorphic $(q-1)$-spheres. 

\begin{thm}
If two Dehn-Sommerville $q$-manifolds $A,B$ are Dehn-Sommerville, then $A \# B$ is 
also a Dehn-Sommerville $q$-manifold. 
\end{thm}
\begin{proof}
Assume the connected sum $A \# B$ has been glued along the sphere $H=S(x) \sim S(y)$. 
If $x \in A \setminus H$ or $y \in B \setminus H$, then the unit sphere in $A \# B$ is the
same than the unit sphere in $A$ or $B$. If $x \in H$, then $S(x)$ consists of 
three parts $\{ y \in S(x) \cap A \setminus H \}, \{ y \in S(x) \cap B \setminus H\}$ and
$\{ y \in S(x) \cap H \}$. But this means that $S(x)$ is essentially the suspension of 
the $q-2$ sphere $S_H(x)$. It is obtained by gluing $S(x) \cap A$ with $S(y) \cap B$ along
$H$. In other words $S(x) = S_A(x) \# S_B(x)$. We can use induction to see that this is a
Dehn-Sommerville sphere. 
\end{proof}

\paragraph{}
An {\bf edge refinement} of a complex picks an edge 
$e=(a,b)$, adds a new vertex $c$ in the middle and connects
$c$ to the bone $S(a) \cap S(b)$.
\index{edge refinement}

\begin{thm}
Edge refinement preserves Dehn-Sommerville manifolds.
\end{thm}
\begin{proof}
The first one is to increase $f_0$ and $f_1$ by $1$ 
(which means adding $t+t^2$ to $f_G(t)$.
Then we add $t f_{S(a) \cap S(b)} + 2 t^2 f_{S(a) \cap S(b)}$, because every
$k$-simplex in $S(a) \cap S(b)$ defines a new $(k+1)$-simplex connecting to $c$ and
two new $(k+2)$-simplices connecting $S(a) \cap S(b)$ to $(a,c)$ and $(b,c)$.
Now, the set of functions satisfying the Dehn-Sommerville symmetry form a linear
space. The claim follows as the added part 
$t+t^2 +t f_{S(a) \cap S(b)} + 2 t^2 f_{S(a) \cap S(b)}$
satisfy the Dehn-Sommerville symmetry by induction because the space $S(a) \cap S(b)$ 
is in $\mathcal{X}_{d-2}$  if $G \in \mathcal{X}_d$. 
\end{proof}

\section{Cohomology}

\paragraph{}
Simplicial complexes $G$ carry a {\bf simplicial cohomology} defined by the 
{\bf exterior derivative} $d$ which naturally exists on functions on $G$. 
If $G$ has $n$ elements, the matrix $d$ is a $n \times n$ matrix. The matrix
depends on a given orientation chosen for each simplex but choosing an order
on each simplex is simply putting a coordinate system. Interesting quantities
like the kernels of $(d+d^*)^2$ do not depend on the coordinate system.
Simplicial cohomology is only the first of a sequence of cohomologies. 
The definitions are reviewed as follows: 

\paragraph{}
If $G$ with $n$ elements, the {\bf exterior derivative}
$d$ is a $n \times n$ matrix. It is defined if every $x \in G$ is totally ordered.
The definition is $d(x,y) = {\rm sign}(x,y)$, if $|x|$ and $|y|$ differ by one. The sign
corresponds to the embedding signature of the permutations. 
The set of {\bf forms} $G \to \mathbb{R}$ identifies with $\mathbb{R}^n$. Functions on 
$p$-simplices are {\bf p-forms}. The linear transformation $d$ maps $p$-forms to $(p+1)$-forms. 
The transpose $d_p^*$ maps $(p+1)$-forms to $p$-forms.
The {\bf Kirchhoff matrix} $d_0^* d_0$ is the analog of the scalar Laplacian as $d_0$ is the analog of 
the {\bf gradient} and $d_0^*$ is the analog the {\bf divergence}.
\index{Kirchhoff matrix}
\index{divergence}
\index{gradient}
\index{exterior derivative}

\paragraph{}
The {\bf Dirac matrix} $D=d+d^*$ defines the {\bf Hodge Laplacian} $L=D^2=(d+d^*)^2=d d^* + d^*d$.
It is block diagonal $L=L_0 \oplus L_1 \oplus \cdots \oplus L_q$. The {\bf Betti vector}
$b = (b_0,b_1, \dots, b_q)$ has as components the {\bf Betti numbers}
$b_p = {\rm dim}({\rm ker}(L_p))$, where $q$ is the {\bf maximal dimension} of $G$.
The linear space of {\bf harmonic $p$-forms} ${\rm ker}(L_p)$ is the {\bf $p$-th cohomology} space.
This is motivated by classical Hodge theory \cite{Hodge1933} 
first adapted to the discrete in \cite{Eckmann1944}. 
The Euler-Poincar\'e formula $\chi(G) = \sum_{k=0}^q (-1)^k b_k(G)$ is a consequence of
McKean-Singer symmetry. We repeat the argument again:
define first the {\bf super trace} of a block diagonal matrix $K=\oplus_{k=0}^q K_k$ as
$\sum_{k=0}^{q} (-1)^k {\rm tr}(K_k)$. Now, 
${\rm str}(e^{-tL})$ is for $t=0$ the super trace ${\rm str}(1)=\chi(G)$
For the limit $t \to \infty$ it is $\sum_{k=0}^q (-1)^k b_k(G)$. 
\index{Hodge Laplacian}
\index{Betti vector}
\index{harmonic forms}
\index{McKean-Singer symmetry}
\index{super trace}
\index{cohomology}

\paragraph{}
The next cohomology is {\bf quadratic cohomology} \cite{valuation,CohomologyWu} which was
also called Wu cohomology as its characteristic $w_2(G)$ is Wu characteristic. 
It defines an exterior derivative on functions on pairs of simplices $(x,y) \in G \times G$.
The {\bf quadratic cohomology} defines a {\bf quadratic Betti vector} $b(G)$ 
which is associated to {\bf quadratic characteristic} 
$w_2(A)=\sum_{(x,y) \in A^2, x \cap y \in A} \omega(x) \omega(y)$ which 
satisfies {\bf energy identities}
$w_2(G)= = \sum_{x \in G^2} \omega(x) \omega(y) w_2(U(x) \cap U(y))$ and 
even has been generalized to k-point Green function identities for all 
{\bf higher order characteristics} \cite{CharacteristicTopologicalInvariants}.
The interaction of open and closed sets produces 
{\bf quadratic fusion inequalities} \cite{fusion2}. This is much 
richer than the {\bf fusion inequality} for simplicial cohomology \cite{fusion1}. 
\index{quadratic cohomology}
\index{higher order characteristic}
\index{fusion inequality}
\index{energy identities}

\paragraph{}
As for simplicial cohomology of Dehn-Sommerville manifolds, we have more 
freedom in building manifolds. There is a basic open ended question for 
simplicial complexes or manifolds: "what Betti vectors are possible for
manifolds up to a given size?" In small dimensions, where the classification of manifolds
is known, we can tell: $1$-manifolds the Betti vector is
$b=(k,k)$, where $k$ is the number of circles, for $2$-manifolds the possible 
Betti vectors are $(a+b,(2a+b)g,b)$, where $a+b \geq 1, b \geq 0,g \geq 0$ are
integers. The reason is that we have
$b=(1,2g,1)$ for a connected orientable manifolds that are a connected sum of $g$ tori,
and $b=(1,g,0)$ for a connected non-orientable manifold that is a 
connected sum of $g$ projective planes. 
In general, there are always Euler Characteristic 
constraints as $\sum_{k=0}^q (-1)^k b_k = \chi(G)$ which is zero for
odd dimensional manifolds. In the orientable manifold case, there is always the Poincar\'e duality 
constraint $b_k=b_{q-k}$ for $k=0, \dots, q$. 

\paragraph{}
Let us end this section with a question:
is it possible that for every simplicial complex $G$ we can construct a Dehn-Sommerville manifold $M$
such that the cohomology of $G$ and $M$ agree? This is affirmative in dimension $1$: for one-dimensional simplicial
complex $G=(V,E)$, the Betti number $b_0$ is the number of connected components and $b_1$ is $1-\chi(G)$
which is $b_1=b_0-|V|+|E|$. For 2-manifolds we have a classification
of 2-manifolds using connected sum construction 
$M=S^2 \# P^1 \# \cdots \# P^1$ for which $b=(1,k)$, where
$k$ is the number of copies of $P^1$. 
There non-orientable 2-manifolds of Euler characteristic $(a,b,0)$ 
for any $a \geq 1, b \geq 0$. 

\paragraph{}
For orientable 2-manifold, the possible Betti vectors are $(a,2b,a)$ with $a \geq 1, b \geq 0$ 
by the classification of orientable 2-manifolds. As for orientable Dehn-Sommerville 2-manifolds, 
the constraints that the middle cohomology is even goes away. Take a graph with Betti vectors $(a,b)$. 
The above construction gives now a orientable Dehn-Sommerville $2$ manifold $G$ such that the 
cohomology is $(a,b,a)$. Just attach suspensions of a circle at every edge. The first cohomology
is now the same as the one from the graph: the possible Betti vectors for graphs are $(a,b,0)$.
By gluing in projective planes rather than 2-spheres, we can realize any Betti vector $(a,b,c)$
with $a \geq 0, b \geq 0, c \geq 0$. 

\paragraph{}
The higher cohomology of Dehn-Sommerville spaces can be very rich. 
Like for q-manifolds, we know that all characteristics of a q-manifold are the same. 
The quadratic cohomology however can already be complicated however. If we construct 
for example from a graph with $|V|=5$ and $|E|=10$ the Dehn-Sommerville 2-manifold,
we have for example $b=(1,6,10)$ for simplicial cohomology and $b=(0,0,45,110,70)$ for 
quadratic cohomology. This needs to be much more explored. 

\section{Connection calculus}

\paragraph{}
In this section we generalize the formula $\omega_m(G)=\chi(G)-\chi(\delta(G))$ from 
manifolds to Dehn-Sommerville manifolds $G$ with boundary $\delta G$. 
In order to do so, we first define 
{\bf Dehn-Sommerville manifolds with boundary}. Inductively, a 
{\bf Dehn-Sommerville q-manifold with boundary}
has either Dehn-Sommerville $(q-1)$-spheres with boundary = Dehn-Sommerville (q-1)-balls 
as unit spheres or then Dehn-Sommerville (q-1)-spheres as unit spheres. 
A $(q-1)$-ball is $H \setminus U(x)$, where $H$ is a $(q-1)$-sphere and $x \in H$. 
We can get examples of such manifolds with boundary using level surface construction.
The point is here that we get natural examples of manifolds with boundary. 

\begin{thm}
Take a function $g: V(G) \to \{0,1, \dots, k\}$ and look at all simplices for which 
$g(x) = \{ 1, \dots k \}$. This is Dehn-Sommerville manifold $q-k+1$ manifold with 
$(q-k)$ manifold as boundary. 
\end{thm}

\paragraph{}
The {\bf connection Laplacian} of $G$ is defined as $L(x,y) = \chi(S^-(x) \cap S^-(y))$. 
By the {\bf unimodularity theorem} \cite{Unimodularity,MukherjeeBera2018,KnillEnergy2020},
this is unimodular matrix with determinant 
$\prod_{x \in G} \omega(x)$. Its inverse $g(x,y)$ is by the {\bf Green star formula} 
$g(x,y) = \omega(x) \omega(y) \chi (S^+(x) \cap S^+(y))$,
where $\omega(x) = (-1)^{{\rm dim}(x)}$. 
The Euler characteristic is $\chi(G) = \sum_{x \in G} \omega(x)$. 
The {\bf Wu characteristic} , the quadratic characteristic, is
$\omega(G) = \sum_{(x,y) \in G^2, x \cap y \neq \emptyset} \omega(x) \omega(y)$. 
\index{Wu characteristic}
\index{quadratic characteristic}

\paragraph{}
Lets start with a lemma about stars $U(x)$: 
\index{star lemma}

\begin{lemma}[Star lemma]
$w_m(U(x)) = \chi(U(x))^m$ for all $m \geq 1$. 
\end{lemma}
\begin{proof}
Just note that any $m$-tuple $X=(x_1, \dots, x_m)$ of simplices
in $U(x)$ all intersect as $x \subset \bigcap X$. 
Therefore $(\sum_{y \in U} \omega(y))^m = w_m(U(x))$. 
\end{proof}

\paragraph{}
\begin{lemma}[Local boundary formula]
If $G$ is a simplicial complex and $x \in G$, then 
$w_m(B(x)) = w_m(U(x)) - (-1)^m w_m(S(x))$ for all $m \geq 1$. 
\end{lemma}
\index{local boundary formula}

\begin{proof} 
We use double induction. The outer induction goes with $m$. For $m=1$ we have
where  $w_1 = \chi$ is the Euler characteristic, we have 
$B(x)=U(x) \cup S(x)$, which is a disjoint union and 
$w_1(B(x)) = w_1(U(x)) + w_1(S(x))$ is the valuation property. \\
Now assume things have been proven for $m-1$. We establish it for $m$. The
inner induction goes now with respect to the size of the simplicial complex.
This can be proven by building up the simplicial complex $G=B_G(x)$
successively adding maximal simplices such that the new complex is always again a ball $H=B_H(x)$.
If $G$ is zero dimensional and $G=B(x)$, then $B(x)=U(x)$ and $S(x)=\emptyset$.
Now add edges until the one dimensional skeleton complex is filled, then add triangles (2-simplices)
etc. Assume we have established the relation for a complex $G=B_G(x)$ and 
we add a new q-simplex $X$ so that $H=B_H(x)$. By construction when adding $X$, the frame
$S(X) \cap B_G(x)$ had already been constructed before. The quantities change as follows:
We have $B_H(x)=B_G(x) \cup \{X\}$ and
        $U_H(x)=U_G(x) \cup \{X\}$ and 
        $S_H(x)=S_G(x) \cup \{Y\}$, where $Y=X \setminus x$.
The subcomplex $K=B_G(X)$ still contains $x$.
All m-tuples in $U_H(x)$ that are not in $U_G(x)$ must contain $X$ as one element and all other elements in $U_K(x)$.
All m-tuples in $B_H(x)$ that are not in $B_G(x)$ must contain $X$ as one element and all other elements in $B_K(x)$.
All m-tuples in $S_H(x)$ that are not in $S_G(x)$ must contain $Y$ as one element and all other elements in $S_K(x)$.
When going from $G$ to $H$ we have added $w_{m-1}(B_K(x)) w_1(X)$ to the left and 
$w_{m-1}(U_K(x)) w_1(X) + w_{m-1}(S_K(x)) w_1(Y)$ to the right. As the dimension of $Y$ is one lower than $X$ and
$w_1(Y)=-w_1(X)$, and having established the relation for $m-1$ (and applying it for $K$), we are done. 
\end{proof} 

\paragraph{}
We illustrate this remarkable relation in examples. \\
a) If $x$ is locally maximal of dimension $q$, 
then $U(x)=\{x\}$ and $w_m(U(x))=(-1)^{q m}$
and $w_m(S(x))=1-(-1)^q$ and $w_m(B(x))=(-1)^{q (m+1)}$. \\
%  (-1)^(q m) - (-1)^m (1-(-1)^q) == (-1)^(q(m+1))
b) For a $1$-dimensional complex, where $x$ is a vertex of degree $d$
we have $S(x)=\overline{K_d}$ and $B(x)$ is the complex of a star graph
and $U(x)$ is a union of $d$ open $K_2$. It is a nice combinatorial exercise
to compute $w_m(G)$ if $G$ is the star graph.  \\
c) For $G=K_2$. $U(x)=\{ \{1\},\{1,2\} \}$, $B(x)=G$ $S(x)=\{ \{2\} \}$. 
$w_2(U)=0$, $w_2(S)=1$, $w_2(B)=-1$. We have $w_2(B(x))=w_2(U(x))-w_2(S(x))$
and $w_1(B(x))=w_2(U(x))+w_2(S(x))$. 
More generally, if $B(x)=K_q,x=K_k,S(x)=K_{q-k}$,
then the left hand side is $(-1)^{m+1}$ and the right $(-1)^{mq}-(-1)^m (-1)^{q-k+1}$.
\begin{comment} G=Whitney[StarGraph[5]]; x={1}; U=OpenStar[G,x]; B=Closure[U]; S=UnitSphere[G,x]; \end{comment} 
d) If $x \in g$ and $G$ is a Dehn-Sommerville manifold then 
$w_m(U(x)) = (-1)^{mq}$ because $\chi(U(x))= (-1)^q$. 
The relation follows from $\chi(B(x))=1$ and $\chi(S(x))=1-(-1)^q$ and the valuation 
property $\chi(S(x))+\chi(U(x))=\chi(B(x))$. 

\paragraph{}
The following result is a generalization of what we know for manifolds.  It is a
global version of the above local analog. Note that the local analog holds for all 
simplicial complexes while global boundary formula that follows needs
a Dehn-Sommerville manifold structure. 
\index{global boundary formula}

\begin{thm}[Global boundary formula]
For a Dehn-Sommerville q-manifold $\omega_m(G)=\chi(G)$ for all $m \geq 1$. 
More generally, for a Dehn-Sommerville manifold with boundary, 
$\omega_m(G) = \chi(G)-(-1)^m \chi(\delta(G))$
\end{thm}

\begin{proof}
Use the local boundary formula. In the case of Dehn-Sommerville manifolds, we have
$w_m(U(x))=(-1)^{qm}$ by the above lemma and
$w_m(S(x))=1-(-1)^q$ by assumption. This now determines 
the $w_m(B(x))=(-1)^{q(m+1)}$. 
% (-1)^(q m) - (-1)^m (1-(-1)^q) == (-1)^(q(m+1))
We have shown that $w_m(G) = \sum_x \omega(x) w_m(U(x))$ and
using the Sphere formula $w_m(G) = \sum_x \omega(x) \omega(x) w_m(B(x))$. 
By induction, we have $w_m(B(x)) = w_1(B(x))-w_1(S(x))$ which is $1 - (1-(-1)^q) = (-1)^q$. 
In the case of Dehn-Sommerville manifolds, we have $w_m(U(x))=(-1)^{qm}$
and $w_m(S(x))=1-(-1)^q$. From the 1-point Green formula and the
sphere formula $0 = \sum_x \omega(x) w_m(S(x))$ we have
$w_m(G) = \sum_x \omega(x) w_m(U(x)) = \sum_x \omega(x) w_m(B(x))$. 
Since $B(x)$ is a smaller manifold with boundary, we have
$\sum_{x \in G} \omega(x)  w_1(B(x)) + (-1)^m w_1(S(x)) 
= \sum_{x \in G} w_1(B(x))$ which by the 1-point Green function 
identity is again $w_1(G)=\chi(G)$. 
\end{proof}

\begin{figure}[!htpb]
\scalebox{0.75}{\includegraphics{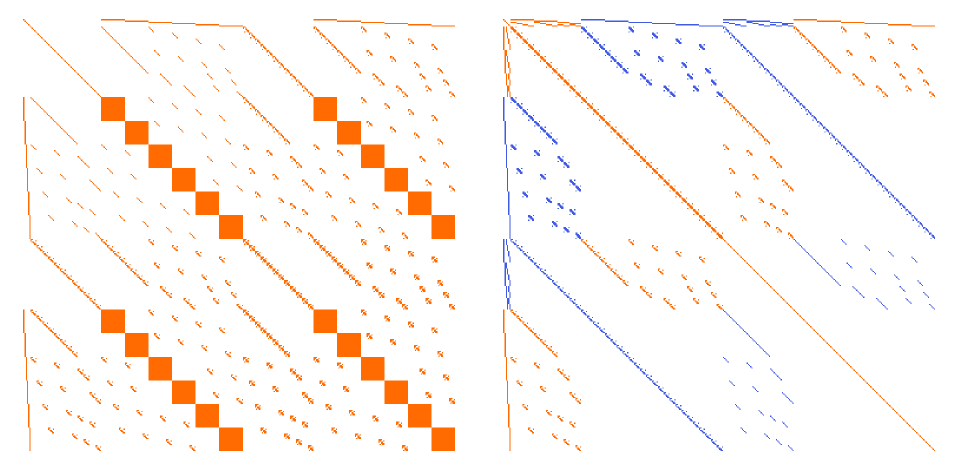}}
\label{Example1}
\caption{ 
The connection Laplacian and its inverse Green function matrix $g=L^{-1}$. 
}
\end{figure}

\begin{comment}
{\bf Wu characteristic} $\omega(G) = \sum_{x \sim y}  \omega(x) \omega(y)$ is
$1-f(-1,-1)$, where 
$$ f(t,s) = 1+\sum_{k,l} f_{kl}(G) t^{k+1} s^{l+1} $$
is the {\bf multivariate simplex generating function}. Here,
$f_{kl}(G)$ is the {\bf $f$-matrix}, counting the number of intersecting $k$-dimensional
and $l$-dimensional simplices.

Old claim: $F_G(t,s) = 1+\int_0^t f(r,s) \; dr$.
$$ f_G(t,s) = \sum_{v \in V} F_{S(v)}(t,s)  \;  $$
and especially $\omega(G) = \sum_{v \in G} K(v)$.

We called $\sum_{x \in G} (1-\chi(S(x))) = {\rm tr}(L-L^{-1})$
the {\bf Hydrogen functional}.  % It was f_G'(-1)   WHY?
Compare $\sum_{x \in G} \omega(x) (1-\chi(S(x))) = \chi(G)$. 
\end{comment}

\section*{Appendix: Graph case} 

\paragraph{}
In our expositions about this subject in 2019 \cite{dehnsommervillegaussbonnet}, we had used the language of graphs, 
rather than the language of simplicial complexes. While complexes coming from complexes are a restriction, 
this is mainly terminology,
because a simplicial complex $G$ defines a graph with $G$ as vertex set where two are connected if one is 
contained in the other. A graph in turn defines a simplicial complex, the Whitney
complex. The composition of these operations is the Barycentric refinement. 
There are complexes that are not Barycentric refinements of an other complex.
Examples are skeleton complexes like the boundary $q-1$ complex of $K_q$
which is a $(q-1)$-sphere simplicial complex but is not the Whitney complex of a 
graph. Let us nevertheless go through the definitions in the case of graphs as
graphs are more intuitive. 

\paragraph{}
The definition of Dehn-Sommerville graph can also be built within the {\bf monoid of 
finite simple graphs}, where the join is the {\bf Zykov join} \cite{Zykov}:
the join $A \oplus B$ of two finite simple graphs is the disjoint union
of the graphs augmented with edges connecting any vertex of $A$ with any
vertex of $B$. The monoid of graphs can be seen as a 
submonoid of the monoid of simplicial complexes because
every {\bf finite simple graph} $\Gamma=(V,E)$ defines a simplicial complex $G$ 
in which the vertex sets of complete subgraphs $K_n$ with 
$n \geq 1$ of $\Gamma$ form the sets of $G$.  $G$ is called the 
{\bf Whitney complex} or {\bf flag complex} or {\bf order complex} of $\Gamma$. 
Since any simplicial complex $G$ defines in turn a graph $\Gamma = (G,\{ (x,y), 
x \subset y \; {\rm or} \; y \subset x \}$ whose simplicial complex $G_1$ is
the Barycentric refinement of $G$, there is hardly any difference between the two 
frame works `simplicial complex" or ``graph". It is mostly a change language. But there
are complexes that are not Whitney complexes. But every {\bf Barycentric refinement}
$G_1$ of any complex $G$ is by definition the Whitney complex of a graph. 

\paragraph{}
Let us just repeat the definition in a graph setting: one says a 
subgraph $(W,F) \subset (V,E)$ {\bf is {\bf induced} by its vertex set $W \subset V$ } if 
$F=\{ (a,b) \in E, a \in W, b\in W \}$. The {\bf unit sphere} of a vertex
$v \in V$ is the graph induced by all neighbor vertices
$\{ w \in V, (v,w) \in E \}$. 
The definition of {\bf Dehn-Sommerville q-manifold graphs}
proceeds in the same way as for simplicial complexes: the empty graph 
$0=\{ \emptyset,\emptyset \}$ is the $(-1)$-sphere graph. A graph is called a 
{\bf Dehn-Sommerville $q$-manifold graph} if for every $v \in V$, the
unit sphere $S(v)$ is a Dehn-Sommerville $(q-1)$-sphere graph for all $v \in V$. 
A graph is a {\bf Dehn-Sommerville $q$-sphere graph }
if it is a Dehn-Sommerville $q$-manifold graph and $\chi((V,E)) = \chi(G) = 1+(-1)^q$.
\index{induced graph}
\index{Dehn-Sommerville manifold graph}
\index{Dehn-Sommerville sphere graph}

\paragraph{}
Given a simplicial complex $G$ we have a graph $(G,E)$ in which the sets in $G$
are vertices and $E=\{ (a,b), a,b \in G, a \subset b$ or $b \subset a) \}$. 
Given a vertex $v \in G$, then the Barycentric refinement of unit sphere $S(v)$ in 
the graph is the unit sphere $S(v)$ in $G$. For a $k$-simplex $x \in G$, the unit sphere $S(x)$
is the join of a $(k-1)$ sphere $S^-(x)$ with a $q-k-1$ sphere $S^+(x)$ which can be seen 
as the intersection $\bigcap_{w \subset x, w \in V}  S(w)$. In principle, we can understand
the geometry of $G$ in terms of the graph geometry of the graph $(G,E)$. 
A Dehn-Sommerville q-manifold $G$ defines a Dehn-Sommerville $q$-manifold graph $(G,E)$. 
On the other hand, a Dehn-Sommerville $q$-manifold graph defines a Dehn-Sommerville
$q$-manifold $G_1$, the Whitney complex of the graph. 

\section*{Appendix: Code} 

\begin{tiny}
\lstset{language=Mathematica} \lstset{frameround=fttt}
\begin{lstlisting}[frame=single]
(* Basics        *)
CleanGraph[s_]:=AdjacencyGraph[AdjacencyMatrix[s]];
CleanComplex[G_]:=Union[Sort[Map[Sort,G]]];
Generate[A_]:=If[A=={},A,CleanComplex[Delete[Union[Sort[Flatten[Map[Subsets,A],1]]],1]]];
Whitney[s_]:=Generate[FindClique[s,Infinity,All]]; Closure=Generate; 
FVector[G_]:=Delete[BinCounts[Map[Length,G]],1];  
Euler[G_]:=Sum[-(-1)^Length[G[[k]]],{k,Length[G]}];
OpenStar[G_,x_]:=Select[G,SubsetQ[#,x] &];
Basis[G_]:=Table[OpenStar[G,G[[k]]],{k,Length[G]}];
UnitSphere[G_,x_]:=Module[{U=OpenStar[G,x]},Complement[Closure[U],U]];
UnitSpheres[G_]:=Table[UnitSphere[G,G[[k]]],{k,Length[G]}];
\end{lstlisting}
\end{tiny}

\begin{tiny}
\lstset{language=Mathematica} \lstset{frameround=fttt}
\begin{lstlisting}[frame=single]
(* Construction *)
RingFromComplex[G_,a_]:=Module[{V=Union[Flatten[G]],n,T,U},
    n=Length[V];Quiet[T=Table[V[[k]]->a[[k]],{k,n}]];
    Quiet[U=G /.T]; Sum[Product[U[[k,l]], {l,Length[U[[k]]]}],{k,Length[U]}]];
ComplexFromRing[f_]:=Module[{s,ff},s={}; ff=Expand[f];
    Do[Do[If[Denominator[ff[[k]]/ff[[l]]]==1 && k!=l,
      s=Append[s,k->l]], {k,Length[ff]}],{l,1,Length[ff]}];
    Whitney[UndirectedGraph[Graph[Range[Length[ff]],s]]]];
GeometricProduct[G_,H_]:=Module[{f,g,F},
    f=RingFromComplex[G,"a"];
    g=RingFromComplex[H,"b"]; F=Expand[f*g]; ComplexFromRing[F]];
TopologicalProduct=GeometricProduct;

ConeExtension[G_]:=Module[{q=Max[Flatten[G]]+1,n=Length[G]},
  Generate[Table[Append[G[[k]],q],{k,n}]]];
Suspension[G_]:=Module[{q=Max[Flatten[G]]+1,n=Length[G]},
  Closure[Union[Table[Append[G[[k]],q  ],{k,n}],
                Table[Append[G[[k]],q+1],{k,n}]]]];
DoubleSuspension[G_]:=Suspension[Suspension[G]];
Addition[A_,B_]:=Module[{q=Max[Flatten[A]],Q},
  Q=Table[B[[k]]+q,{k,Length[B]}];Sort[Union[A,Q]]];
JoinAddition[G_,H_]:=Union[G,H+Max[G]+1,
                   Map[Flatten,Map[Union,Flatten[Tuples[{G,H+Max[G]+1}],0]]]];

AddTwoSphere[s_]:=Module[{e,x,m,v=VertexList[s]},x=First[v]; e=EdgeList[s]; m=Length[v];
   UndirectedGraph[Graph[Union[e,{x->m+1,m+1->m+2,m+2->m+3,m+3->x,
   m+1->m+4,m+4->m+3,m+3->m+5, m+5->m+1,x->m+4,m+4->m+2,m+2->m+5,m+5->x}]]]];
Bouquet[n_]:=Nest[AddTwoSphere,CompleteGraph[{2,2,2}],n];

Shannon[A_,B_]:=Module[{q=Max[Flatten[A]],Q,G={}},Q=Table[B[[k]]+q,{k,Length[B]}];
  Do[G=Append[G,Sort[Union[A[[a]],Q[[b]]]]],{a,Length[A]},{b,Length[Q]}];
  If[A=={},G={}]; If[B=={},G={}]; Sort[G]];

\end{lstlisting}
\end{tiny}

\begin{tiny}
\lstset{language=Mathematica} \lstset{frameround=fttt}
\begin{lstlisting}[frame=single]
(* Level sets *)

R[G_,k_]:=Module[{},R[x_]:=x->RandomChoice[Range[k]]; Map[R,Union[Flatten[G]]]];
Surface[G_,g_]:=Select[G,SubsetQ[#/.g,Union[Flatten[G] /. g]] &];
S[s_,v_]:=VertexDelete[NeighborhoodGraph[s,v],v];     Sf[s_,v_]:=FVector[Whitney[S[s,v]]];
Curvature[s_,v_]:=Module[{f=Sf[s,v]},1+f.Table[(-1)^k/(k+1),{k,Length[f]}]];
Curvatures[s_]:=Module[{V=VertexList[s]},Table[Curvature[s,V[[k]]],{k,Length[V]}]];
ToGraph[G_]:=UndirectedGraph[n=Length[G];Graph[Range[n],
  Select[Flatten[Table[k->l,{k,n},{l,k+1,n}],1],(SubsetQ[G[[#[[2]]]],G[[#[[1]]]]])&]]];
Barycentric[s_]:=ToGraph[Whitney[s]];
\end{lstlisting}
\end{tiny}

\begin{tiny}
\lstset{language=Mathematica} \lstset{frameround=fttt}
\begin{lstlisting}[frame=single]
(* Cohomology *)

Coho[G_]:=Module[{n,q,f,d,Dirac,U,H},
  n=Length[G]; q=Map[Length,G]-1;f=Delete[BinCounts[q],1];
  Orient[a_,b_]:=Module[{z,c,k=Length[a],l=Length[b]}, If[SubsetQ[a,b] &&
  (k==l+1),z=Complement[a,b][[1]];c=Prepend[b,z];Signature[a]*Signature[c],0]];
  d=Table[0,{n},{n}]; d=Table[Orient[G[[i]],G[[j]]],{i,n},{j,n}];
  Dirac=d+Transpose[d]; H=Dirac.Dirac; f=Prepend[f,0]; m=Length[f]-1;
  U=Table[v=f[[k+1]];Table[u=Sum[f[[l]],{l,k}];H[[u+i,u+j]],{i,v},{j,v}],{k,m}];
  Map[NullSpace,U]];
Betti[G_]:=Map[Length,Coho[G]];

Coho2[G_,H_]:=Module[{n=Length[G],m=Length[H],U={},d1,d2}, len[x_]:=Total[Map[Length,x]];
  Do[If[Length[Intersection[G[[i]],H[[j]]]]>0,U=Append[U,{G[[i]],H[[j]]}]],{i,n},{j,m}];
  U=Sort[U,len[#1]<len[#2] & ];u=Length[U];l=Map[len,U]; w=Union[l];
  b=Prepend[Table[Max[Flatten[Position[l,w[[k]]]]],{k,Length[w]}],0]; h=Length[b]-1;
  deriv1[{x_,y_}]:=Table[{Sort[Delete[x,k]],y},{k,Length[x]}];
  deriv2[{x_,y_}]:=Table[{x,Sort[Delete[y,k]]},{k,Length[y]}];
  d1=Table[0,{u},{u}]; Do[v=deriv1[U[[m]]]; If[Length[v]>0,
    Do[r=Position[U,v[[k]]]; If[r!={},d1[[m,r[[1,1]]]]=(-1)^k],{k,Length[v]}]],{m,u}];
  d2=Table[0,{u},{u}]; Do[v=deriv2[U[[m]]]; If[Length[v]>0,
    Do[r=Position[U,v[[k]]]; If[r!={},d2[[m,r[[1,1]]]]=(-1)^(Length[U[[m,1]]]+k)],
    {k,Length[v]}]],{m,u}]; d=d1+d2; Dirac=d+Transpose[d]; L=Dirac.Dirac; Map[NullSpace,
  Table[Table[L[[b[[k]]+i,b[[k]]+j]],{i,b[[k+1]]-b[[k]]},{j,b[[k+1]]-b[[k]]}],{k,h}]]];
Betti2[G_,H_]:=Map[Length,Coho2[G,H]];Coho2[G_]:=Coho2[G,G]; Betti2[G_]:=Betti2[G,G];
\end{lstlisting}
\end{tiny}

\begin{tiny}
\lstset{language=Mathematica} \lstset{frameround=fttt}
\begin{lstlisting}[frame=single]
(* Invariants *)

X[k_,q_]:=Module[{f},f[i_]:=Table[If[i==j-1,1,0],{j,q+1}];
  Sum[(-1)^(j+q) Binomial[j+1,k+1] f[j],{j,k,q}]- f[k]];
BarycentricOperator[q_]:=Table[StirlingS2[j,i]i!,{i,q+1},{j,q+1}];
Y[k_,q_]:=Reverse[Eigenvectors[Transpose[BarycentricOperator[q]]]][[k+1]];

Ffunction[G_]:=Module[{f=FVector[G],n},Clear[s]; 1+Sum[f[[k]]*s^k,{k,Length[f]}]];
DehnSommervilleQ[G_]:=Module[{f},Clear[s];f=Ffunction[G]; 
   Simplify[f] === Simplify[(f /. s->-1-s)]];
\end{lstlisting}
\end{tiny}

\begin{tiny}
\lstset{language=Mathematica} \lstset{frameround=fttt}
\begin{lstlisting}[frame=single]
(* Connection Stuff *)

ConnectionLaplacian[G_]:=Table[If[MemberQ[G,Sort[Intersection[G[[i]],G[[j]]]]],1,0],
  {i,Length[G]},{j,Length[G]}];
ConnectionGreenFunction[G_]:=Inverse[ConnectionLaplacian[G]];

w[x_]:=-(-1)^Length[x];  dim[x_]:=Length[x]-1;
Wu1[A_]:=Total[Map[w,A]];
Wu2[A_]:=Module[{a=Length[A]},Sum[x=A[[k]]; Sum[y=A[[l]];
   If[MemberQ[A,Intersection[x,y]],1,0]*w[x]*w[y],{l,a}],{k,a}]];Wu=Wu2;
Wu3[A_]:=Module[{a=Length[A]},Sum[x=A[[k]];Sum[y=A[[l]];Sum[z=A[[o]];
   If[MemberQ[A,Intersection[x,y,z]],1,0]*w[x]*w[y]*w[z],{o,a}],{l,a}],{k,a}]];
\end{lstlisting}
\end{tiny}

\printindex

\bibliographystyle{plain}

\begin{thebibliography}{10}

\bibitem{AlbertsonStromquist}
M.O. Albertson and W.R. Stromquist.
\newblock Locally planar toroidal graphs are {$5$}-colorable.
\newblock {\em Proc. Amer. Math. Soc.}, 84(3):449--457, 1982.

\bibitem{Alexandroff1937}
P.~Alexandroff.
\newblock Diskrete {R\"aume}.
\newblock {\em Mat. Sb. 2}, 2, 1937.

\bibitem{Barnette1973}
D.~Barnette.
\newblock The triangulations of the 3-sphere with up to 8 vertices.
\newblock {\em Journal of Combinatorial Theory (A)}, pages 37--52, 1973.

\bibitem{BayerBillera}
M.~Bayer and L.J. Billera.
\newblock Generalized {D}ehn-{S}ommerville relations for polytopes, spheres and
  {E}ulerian partially ordered sets.
\newblock {\em Invent. Math.}, 79(1):143--157, 1985.

\bibitem{BergerLadder}
M.~Berger.
\newblock {\em Jacob's Ladder of Differential Geometry}.
\newblock Springer Verlag, Berlin, 2009.

\bibitem{BrentiWelker}
F.~Brenti and V.~Welker.
\newblock {$f$}-vectors of barycentric subdivisions.
\newblock {\em Math. Z.}, 259(4):849--865, 2008.

\bibitem{ChartrandKronkWall1968}
G.~Chartrand, H.V. Kronk, and C.E. Wall.
\newblock The point-arboricity of a graph.
\newblock {\em Israel Journal of Mathematics}, 6:169--175, 1968.

\bibitem{Dehn1905}
M.~Dehn.
\newblock Die eulersche formel im zusammenhang mit dem inhalt in der
  nicht-euklidischen geometrie.
\newblock {\em Math. Ann.}, 61:561, 1905.

\bibitem{DehnHeegaard}
M.~Dehn and P.~Heegaard.
\newblock Analysis situs.
\newblock {\em Enzyklopaedie d. Math. Wiss}, III.1.1:153--220, 1907.

\bibitem{Eckmann1944}
B.~Eckmann.
\newblock {Harmonische Funktionen und Randwertaufgaben in einem Komplex}.
\newblock {\em Comment. Math. Helv.}, 17(1):240--255, 1944.

\bibitem{Gruenbaum1970}
B.~Gr{\"u}nbaum.
\newblock Polytopes, graphs, and complexes.
\newblock {\em Bull. Amer. Math. Soc.}, 76:1131--1201, 1970.

\bibitem{Gruenbaum2003}
B.~Gr{\"u}nbaum.
\newblock Are your polyhedra the same as my polyhedra?
\newblock In {\em Discrete and computational geometry}, volume~25 of {\em
  Algorithms Combin.}, pages 461--488. Springer, Berlin, 2003.

\bibitem{gruenbaum}
B.~Gr\"unbaum.
\newblock {\em Convex Polytopes}.
\newblock Springer, 2003.

\bibitem{Hetyei}
G.~Hetyei.
\newblock The {S}tirling polynomial of a simplicial complex.
\newblock {\em Discrete and Computational Geometry}, 35:437--455, 2006.

\bibitem{Hodge1933}
W.F.D. Hodge.
\newblock Harmonic functionals in a riemannian manifold.
\newblock {\em Proc. London Math. Soc}, 36:257--303, 1933.

\bibitem{Klain2002}
D.~Klain.
\newblock Dehn-{S}ommerville relations for triangulated manifolds.
\newblock {{\\} http://faculty.uml.edu/dklain/ds.pdf}, 2002.

\bibitem{KlainRota}
D.A. Klain and G-C. Rota.
\newblock {\em Introduction to geometric probability}.
\newblock Lezioni Lincee. Accademia nazionale dei lincei, 1997.

\bibitem{Klee1964}
V.~Klee.
\newblock A combinatorial analogue of {P}oincar{\'e}'s duality theorem.
\newblock {\em Canadian J. Math.}, 16:517--531, 1964.

\bibitem{cherngaussbonnet}
O.~Knill.
\newblock A graph theoretical {Gauss-Bonnet-Chern} theorem.
\newblock {\\}http://arxiv.org/abs/1111.5395, 2011.

\bibitem{poincarehopf}
O.~Knill.
\newblock A graph theoretical {Poincar\'e-Hopf} theorem.
\newblock {\\} http://arxiv.org/abs/1201.1162, 2012.

\bibitem{indexformula}
O.~Knill.
\newblock An index formula for simple graphs \hfill.
\newblock {\\}http://arxiv.org/abs/1205.0306, 2012.

\bibitem{indexexpectation}
O.~Knill.
\newblock On index expectation and curvature for networks.
\newblock {\\}http://arxiv.org/abs/1202.4514, 2012.

\bibitem{KnillTopology}
O.~Knill.
\newblock A notion of graph homeomorphism.
\newblock {{\\}http://arxiv.org/abs/1401.2819}, 2014.

\bibitem{KnillSard}
O.~Knill.
\newblock A {S}ard theorem for graph theory.
\newblock {{\\}http://arxiv.org/abs/1508.05657}, 2015.

\bibitem{valuation}
O.~Knill.
\newblock Gauss-{B}onnet for multi-linear valuations.
\newblock {\\}http://arxiv.org/abs/1601.04533, 2016.

\bibitem{Unimodularity}
O.~Knill.
\newblock On {F}redholm determinants in topology.
\newblock {\\}https://arxiv.org/abs/1612.08229, 2016.

\bibitem{CohomologyWu}
O.~Knill.
\newblock The cohomology for {W}u characteristics.
\newblock {\\}http://arxiv.org/abs/1803.06788, 2017.

\bibitem{HamiltonianManifolds}
O.~Knill.
\newblock Combinatorial manifolds are hamiltonian.
\newblock {\\}https://arxiv.org/abs/1806.06436, 2018.

\bibitem{dehnsommervillegaussbonnet}
O.~Knill.
\newblock Dehn-{S}ommerville from {G}auss-{B}onnet.
\newblock {\\}https://arxiv.org/abs/1905.04831, 2019.

\bibitem{EnergizedSimplicialComplexes}
O.~Knill.
\newblock Energized simplicial complexes.
\newblock https://arxiv.org/abs/1908.06563, 2019.

\bibitem{MorePoincareHopf}
O.~Knill.
\newblock {More on Poincar\'e-Hopf and Gauss-Bonnet}.
\newblock https://arxiv.org/abs/1912.00577, 2019.

\bibitem{PoincareHopfVectorFields}
O.~Knill.
\newblock Poincar{\'e}-{H}opf for vector fields on graphs.
\newblock {\\}https://arxiv.org/abs/1911.04208, 2019.

\bibitem{EnergizedSimplicialComplexes2}
O.~Knill.
\newblock Division algebra valued energized simplicial complexes.
\newblock https://arxiv.org/abs/2008.10176, 2020.

\bibitem{KnillEnergy2020}
O.~Knill.
\newblock The energy of a simplicial complex.
\newblock {\em Linear Algebra and its Applications}, 600:96--129, 2020.

\bibitem{GreenFunctionsEnergized}
O.~Knill.
\newblock Green functions of energized complexes.
\newblock {\\}https://arxiv.org/abs/2010.09152, 2020.

\bibitem{EnergizedSimplicialComplexes3}
O.~Knill.
\newblock Green functions of energized complexes.
\newblock https://arxiv.org/abs/2010.09152, 2020.

\bibitem{CharacteristicTopologicalInvariants}
O.~Knill.
\newblock Characteristic topological invariants.
\newblock https://arxiv.org/abs/2302.02510, 2023.

\bibitem{DiscreteAlgebraicSets}
O.~Knill.
\newblock Discrete algebraic sets in discrete manifolds.
\newblock https://arxiv.org/abs/2312.14671, 2023.

\bibitem{KnillTopology2023}
O.~Knill.
\newblock Finite topologies for finite geometries.
\newblock {{\\}http://arxiv.org/abs/2301.03156}, 2023.

\bibitem{Sphereformula}
O.~Knill.
\newblock The sphere formula.
\newblock https://arxiv.org/abs/2301.05736, 2023.

\bibitem{KnillCuisenaire}
O.~Knill.
\newblock The colorful ring of partitions.
\newblock 2024.

\bibitem{fusion2}
O.~Knill.
\newblock Fusion inequality for quadratic cohomology, 2024.

\bibitem{DiscreteAlgebraicSets2}
O.~Knill.
\newblock Manifolds from partitions.
\newblock https://arxiv.org/abs/2401.07435, 2024.

\bibitem{geodesics1}
O.~Knill.
\newblock Geodesics for discrete manifolds, 2025.

\bibitem{geodesics2}
O.~Knill.
\newblock Interacting geodesics on discrete manifolds, 2025.

\bibitem{fusion1}
Oliver Knill.
\newblock Cohomology of open sets, 2023.

\bibitem{lakatos}
I.~Lakatos.
\newblock {\em Proofs and Refutations}.
\newblock Cambridge University Press, 1976.

\bibitem{LuzonMoron}
A.~Luzon and M.A. Moron.
\newblock Pascal triangle, {S}tirling numbers and the unique invariance of the
  euler characteristic.
\newblock arxiv.1202.0663, 2012.

\bibitem{May2008}
J.P. May.
\newblock Finite topological spaces.
\newblock Notes for REU, Chicago, 2003-2008, 2008.

\bibitem{MukherjeeBera2018}
S.K. Mukherjee and S.~Bera.
\newblock A simple elementary proof of {The Unimodularity Theorem} of {O}liver
  {K}nill.
\newblock {\em Linear Algebra and Its applications}, pages 124--127, 2018.

\bibitem{MuraiNovik}
S.~Murai and I.~Novik.
\newblock Face numbers of manifolds with boundary.
\newblock http://arxiv.org/abs/1509.05115, 2015.

\bibitem{Mycielski}
J.~Mycielski.
\newblock Sur le coloriage des graphs.
\newblock {\em Colloq. Math.}, 3:161--162, 1955.

\bibitem{NovikSwartz}
I.~Novik and E.~Swartz.
\newblock Applications of {K}lee's {D}ehn-{S}ommerville relations.
\newblock {\em Discrete Comput. Geom.}, 42(2):261--276, 2009.

\bibitem{Sommerville1927}
D.~Sommerville.
\newblock The relations connecting the angle sums and volume of a polytope in
  space of n dimensions.
\newblock {\em Proceedings of the Royal Society, Series A}, 115:103--119, 1927.

\bibitem{Wu1953}
Wu~W-T.
\newblock Topological invariants of new type of finite polyhedrons.
\newblock {\em Acta Math. Sinica}, 3:261--290, 1953.

\bibitem{Ziegler}
G.M. Ziegler.
\newblock {\em Lectures on Polytopes}.
\newblock Springer Verlag, 1995.

\bibitem{Zykov}
A.A. Zykov.
\newblock On some properties of linear complexes. ({R}ussian).
\newblock {\em Mat. Sbornik N.S.}, 24(66):163--188, 1949.

\end{thebibliography}

\end{document}